\documentclass[10pt,oneside]{article}
\usepackage[utf8]{inputenc}
\usepackage{amsthm,mathtools}
\usepackage{amsfonts}
\usepackage[margin=3cm]{geometry}

\usepackage[bookmarks,colorlinks]{hyperref}
\hypersetup{colorlinks,citecolor={blue}}

\theoremstyle{plain}
\newtheorem{prop}{Proposition}[section]
\newtheorem{theorem}[prop]{Theorem}
\newtheorem{lemma}[prop]{Lemma}

\theoremstyle{definition}
\newtheorem{definition}[prop]{Definition}

\theoremstyle{remark}
\newtheorem{remark}[prop]{Remark}

\numberwithin{equation}{section}

\title{\textbf{The Einstein condition on nearly K\"ahler six-manifolds}}
\author{Giovanni Russo}
\date{}

\DeclareFontFamily{U}{bigeuf}{}
\DeclareFontShape{U}{bigeuf}{m}{n}{<-6>s*[1.5]eufm5
<6-8>s*[1.5]eufm7
<8->s*[1.5]eufm10}{}
\DeclareSymbolFont{bigeufletters}{U}{bigeuf}{m}{n}
\DeclareMathSymbol{\cyclicsum}{\mathop}{bigeufletters}{`S}

\newcommand{\Ric}{\mathrm{Ric}}
\newcommand{\Ricc}{\mathrm{Ric}^*}

\DeclareMathOperator{\id}{Id}

\begin{document}

\maketitle

\begin{abstract}
We review basic facts on the structure of nearly K\"ahler manifolds, focussing in particular on the six-dimensional case.
A self-contained proof that nearly K\"ahler six-manifolds are Einstein is given by combining different known results.
We finally rephrase the definition of nearly K\"ahler six-manifold in terms of a pair of partial differential equations.
\end{abstract}


\tableofcontents

\section{Introduction}
\label{introduction}

An almost Hermitian geometry is a triple $(M,g,J)$, where $M$ is a $2n$-dimensional manifold equipped with a
Riemannian metric $g$ and an orthogonal almost complex structure $J$. 
Denote by $\nabla$ the Levi-Civita connection on $M$.
Lowering the upper index of $J$ yields the fundamental two-form
$\sigma \coloneqq g(J{}\cdot{},{}\cdot{})$. 
Each tangent space is then a $\mathrm{U}(n)$-module isomorphic to a copy of $\mathbb{C}^n$ with its standard $\mathrm{U}(n)$-structure. 
In a paper published in 1980, Gray and Hervella \cite{gray_hervella} showed how to classify such geometries.
Take a Euclidean, $2n$-dimensional vector space $(V,g_0)$ equipped with an orthogonal complex structure $J_0$. 
The triple $(V,g_0,J_0)$ models each tangent space of $M$.
Define $\mathcal{W}$ as the vector space of the type $(3,0)$-tensors on $V$ satisfying the same symmetries of $\nabla \sigma$. 
Using the notation as in Salamon's book \cite[Chapter 3]{salamon},
\begin{equation*}
\mathcal{W} \coloneqq \Lambda^1 \otimes [\![\Lambda^{2,0}]\!],
\end{equation*}
where $[\![\Lambda^{2,0}]\!]$ is the eigenspace of $J_0$ in $\Lambda^2V^*$ associated with the eigenvalue~$-1$ and $\Lambda^1$ stands for $V^*$. 
In general the space $\mathcal{W}$ splits under the action of the unitary group $\mathrm{U}(n)$ into the orthogonal
direct sum of four irreducible submodules:
\begin{equation*}
\mathcal{W} = \mathcal{W}_1 \oplus \mathcal{W}_2 \oplus \mathcal{W}_3 \oplus \mathcal{W}_4.
\end{equation*} 
Consequently, $\nabla \sigma \in \mathcal{W}$ may be decomposed accordingly. 
Different combinations of its components determine sixteen classes of geometries.
A trivial example are K\"ahler manifolds, which are obtained when $\nabla \sigma = 0$, or equivalently $\nabla J =0$.

In this work we focus on $\mathcal{W}_1$, the class of \emph{nearly K\"ahler} manifolds. 
Their formal definition was given by Gray in the 1970s.

\begin{definition}[Gray \cite{gray1976}]
\label{nk_definition}
Let $(M,g,J)$ be an almost Hermitian manifold with Riemannian metric $g$ and almost complex structure $J$ compatible with $g$. 
Let $\nabla$ denote the Levi-Civita connection on $M$. 
Then $M$ is called \textit{nearly K\"ahler} if $(\nabla_X J)X = 0$ for every vector field $X$ on $M$.
\end{definition}
In the literature one often finds the expression \emph{strict} nearly K\"ahler for nearly K\"ahler manifolds that are not K\"ahler, namely $\nabla J \not \equiv 0$. 
When we write ``nearly K\"ahler'' we mean in fact ``strict nearly K\"ahler'', so as to simplify the terminology. Our manifolds will be always assumed to be connected.

An important ingredient in the general structure theory is the six-dimensional case, as shown by Nagy \cite[Theorem 1.1]{nagy}: 
any complete nearly K\"ahler manifold is locally a Riemannian product of homogenous nearly K\"ahler spaces, twistor spaces over quaternionic K\"ahler manifolds, and six-dimensional nearly K\"ahler manifolds.
On the other hand the classification by Gray and Hervella tells us $\mathcal{W}_1$ is trivial
in dimension two and four (see also the note by Gray \cite[Lemma 3]{gray1969}). Therefore we focus on the case where $M$ has dimension six. 
The lack of explicit six-dimensional, compact examples and the outstanding difficulty in finding new ones has made this geometry particularly exotic and appealing. 
There are only four homogeneous, compact examples (see \cite{frolicher}, \cite{fukami_ishihara} on the homogeneous nearly K\"ahler structure on the six-sphere, and \cite{gray_wolf}, \cite{butruille} for a classification of the homogeneous examples).
In 2017, Foscolo and Haskins \cite{foscolo_haskins} proved the existence of the first non-homogeneous nearly K\"ahler structures in dimension six. Progress in the theory of nearly K\"ahler six-manifolds with two-torus symmetry was made quite recently in~\cite{russo_swann}, where a new explicit, non-compact example is also given.

Certainly interest in nearly K\"ahler structures stems from other facts as well, e.g.\ links with $\mathrm{G}_2$ and spin geometry (see \cite{bryant}, \cite{friedrich_grunewald}, \cite{bar}). One may refer to \cite{phd_thesis} for a comprehensive survey. Here we concentrate on
\begin{theorem}
\label{einstein_condition}
Nearly K\"ahler six-manifolds are Einstein with positive scalar curvature.
\end{theorem}
This is a deep result proved first by Gray in 1976 \cite[Theorem 5.2]{gray1976} and then studied again by Carri\'on \cite{carrion}, Morris \cite{morris}. Friedrich and Grunewald \cite{friedrich_grunewald}, \cite{grunewald} proved that on nearly K\"ahler six-manifolds there exists a Killing spinor, which implies Theorem \ref{einstein_condition}. In \cite{carrion} it is shown that in dimension six there is an equivalence between Definition \ref{nk_definition} and a system of PDEs in terms of an $\mathrm{SU}(3)$-structure, whence 
\begin{theorem}
\label{theorem_carrion}
Let $(M,g,J)$ be an almost Hermitian six-manifold and $\sigma$ be the fundamental two-form on $M$. Then $M$ is nearly K\"ahler if and only if there exist a constant function $\mu$ on $M$ and a complex
$(3,0)$-form $\psi_{\mathbb{C}} = \psi_++i\psi_-$ such that 
\begin{equation}
\label{nk_structure_equations}
d\sigma = 3\mu \psi_+, \qquad d\psi_- = -2\mu \sigma \wedge \sigma.
\end{equation}
\end{theorem}

The function $\mu$ appears when computing the norm of any vector field of the form $(\nabla_XJ)Y$, in fact we will show exactly how. 
That $\mu$ is constant is closely related to the Einstein condition. This seems to be a delicate issue in the literature. 
The main ideas behind it are scattered in essentially two works by Gray \cite{gray}, \cite{gray1976}, but the massive---though impressive---amount of technical formulas obscures the key steps. 
A clearer approach was pursued by Morris, who nonetheless seems to gloss over the details of a crucial step (see in particular the proof of formula ($4.15$), Section $4.2$ in \cite{morris}, where there is no explanation of the fact that $\beta$ sits inside $\mathrm{Sym}^2(\Lambda^2 TM^*)$). What is more, Carri\'on discusses the equivalence stated in Theorem \ref{theorem_carrion}, but does not provide any direct proof of the fact that $\mu$ is constant: this point is claimed to be a consequence of \cite[Theorem 4.20]{carrion}. Notice that if one assumes to have an $\mathrm{SU}(3)$-structure $(\sigma,\psi_{\pm})$ on the manifold satisfying Equations \eqref{nk_structure_equations}, the constancy of $\mu$ is then immediate: since $\sigma \wedge \psi_+=0$ ($\sigma$ is of type $(1,1)$ with respect to the almost complex structure defined by $\psi_{\mathbb C}$) we calculate
$$0 = d(d\psi_-) = -2d\mu \wedge \sigma \wedge \sigma - 12 \mu^2 \sigma \wedge \psi_+ = -2d\mu \wedge \sigma^2,$$
whence $d\mu = 0$, as $\sigma$ is non-degenerate. Part of our work will be to show that $\mu$ is constant using only the $\mathrm{U}(3)$-structure on $M$. We will see how this fact leads to Theorem \ref{einstein_condition} and \ref{theorem_carrion}.

The goal of this paper is to go again through the proof that nearly K\"ahler metrics in dimension six are Einstein, hoping to provide people interested in this field with a unifying reference.
We keep a basic approach, in particular we do not make use of the existence of Killing spinors. We describe symmetries and introduce useful curvature identities combining the results of Gray and Morris', thus giving a complete proof of Theorem \ref{einstein_condition}. 
Finally, we expand the outcome with a self-contained proof of Theorem~\ref{theorem_carrion}, hence rephrase Definition \ref{nk_definition} in terms of the PDEs \eqref{nk_structure_equations}. \\

\textbf{Acknowledgements}. The material contained in this paper is part of my PhD thesis \cite{phd_thesis}. This work was partly supported by the Danish Council for Independent Research | Natural Sciences 
Project DFF - 6108\kern .05em-00358, and by the Danish National Research Foundation grant DNRF95 (Centre for Quantum Geometry of Moduli Spaces). 
I am also grateful to the Max Planck Institute for Mathematics in Bonn for its hospitality and financial support. Further, I thank the anonymous reviewer for useful comments on the exposition.
Lastly, I thank Andrew Swann for his effective, crucial help along this part of my PhD project.

\section{Symmetries}
\label{symmetries}

Let us start off with a nearly K\"ahler manifold $(M,g,J)$ as in Definition \ref{nk_definition}. We assume throughout $M$ to be connected. 
This is not restrictive and will simplify some parts of the exposition. Define the two-form $\sigma \coloneqq g(J{}\cdot{},{}\cdot{})$ and let $\nabla$ be the Levi-Civita connection. We work without specifying the dimension of $M$ and switch to the six-dimensional case only when needed. 

\begin{lemma}
\label{covariant_derivative_sigma_etc}
Let $(M,g,J,\sigma)$ be a nearly K\"ahler manifold and $X,Y,Z$ be vector fields on $M$.
\begin{enumerate}
\item We have the general formula
\begin{equation}
\label{covariant_derivative_j}
(\nabla_{JX}J)Y = -J(\nabla_XJ)Y.
\end{equation}
\item We can move $J$ across all the entries of $\nabla\sigma$:
\begin{equation}
\label{moving_j_around_eq}
\nabla\sigma(JX,Y,Z) = \nabla\sigma(X,JY,Z) = \nabla\sigma(X,Y,JZ).
\end{equation}
\end{enumerate}
\end{lemma}
\begin{proof}
The first formula follows from the identity $0=(\nabla J^2) = (\nabla J)J+J(\nabla J)$ and the skew-symmetry of $\nabla J$:
\begin{equation*}
(\nabla_{JX}J)Y = -(\nabla_YJ)JX = J(\nabla_YJ)X = -J(\nabla_XJ)Y.
\end{equation*} 
Now recall that $g$ is $\nabla$-parallel, so for each triple $U,V,Z$ of vector fields one has
\begin{align*}
\nabla \sigma (U,V,Z) & = U(g(JV,Z))-g(J\nabla_UV,Z)-g(JV,\nabla_UZ) \nonumber \\
& = g(\nabla_U JV,Z)-g(J\nabla_UV,Z) \nonumber \\
& = g((\nabla_UJ)V,Z).
\end{align*}
Hence $\nabla\sigma(JX,Y,Z) = g((\nabla_{JX}J)Y,Z) = -g(J(\nabla_XJ)Y,Z)$. But $J$ is orthogonal, thus the latter equals $\nabla\sigma(X,Y,JZ)$. On the other hand $-g(J(\nabla_XJ)Y,Z)$ coincides with $g((\nabla_XJ)JY,Z)=\nabla\sigma(X,JY,Z)$ as well.
\end{proof}

\begin{remark}
\label{remark_1}
Since $M$ is nearly K\"ahler, $\nabla \sigma$ is a three-form, and $(\nabla_X J)Y$ is orthogonal to $X,Y,JX,JY$. Conversely, if we assume $\nabla \sigma$ to be skew-symmetric, then $\nabla\sigma(X,X,Y) = g((\nabla_X J)X,Y)=0$ for every $Y$, so $M$ is nearly K\"ahler. 
\end{remark}

For the rest of this section we assume $M$ has dimension six. The main intention here is to provide a unifying language to describe symmetries of useful tensors. We refer to Salamon~\cite{salamon} for notations and ideas.
Recall the identity of Lie groups 
\begin{equation*}
\mathrm{U}(n) = \mathrm{SO}(2n) \cap \mathrm{GL}(n,\mathbb{C}).
\end{equation*}
In real dimension six this tells us $\mathrm{U}(3)$ is the stabiliser in $\mathrm{GL}(6,\mathbb{R})$ of an inner product and a complex structure $J_0$ on a copy of $\mathbb{R}^6$. At the level of Lie algebras, this identity implies in particular that elements of $\mathfrak{u}(3)$ commute with~$J_0$. We shall always think of $\mathrm{U}(3)$ as a subgroup of $\mathrm{SO}(6)$. At each point of $M$ there is a representation of $\mathrm{U}(3)$ on the tangent space inducing the structure of $\mathrm{U}(3)$-module on the complexified vector space of $k$-forms, which we denote simply by $\Lambda^k \otimes \mathbb{C}$. Note that every orthogonal matrix coincides with the transpose of its inverse, so the $\mathrm{U}(3)$-modules $\Lambda^k T_pM{}^*$ and $\Lambda^k T_pM$ are equivalent, and one loses no information in identifying $k$-forms and $k$-vectors. This explains the choice of the symbol $\Lambda^k$ for the space of real $k$-forms, and will allow us to identify $\mathrm{U}(3)$-modules and their duals in other circumstances. There is an isomorphism of vector spaces 
$$\Lambda^k \otimes \mathbb{C} = \bigoplus_{p+q=k} \Lambda^p(\Lambda^{1,0}) \otimes \Lambda^q(\Lambda^{0,1}),$$ 
and by definition $\Lambda^{p,q} \coloneqq \Lambda^p(\Lambda^{1,0}) \otimes \Lambda^q(\Lambda^{0,1})$ is the space of complex differential forms of type $(p,q)$. Each $\Lambda^{p,q}$ is a $\mathrm{U}(3)$-invariant complex module.

For $p\neq q$ we denote by $[\![\Lambda^{p,q}]\!]$ the \emph{real} vector space underlying $\Lambda^{p,q}$, whose complexification is $[\![\Lambda^{p,q}]\!] \otimes \mathbb{C} = \Lambda^{p,q} \oplus \Lambda^{q,p}$, whereas $[\Lambda^{p,p}]$ is the space of type $(p,p)$-forms $\alpha$ such that $\overline{\alpha} = \alpha$, hence $[\Lambda^{p,p}] \otimes \mathbb{C}=\Lambda^{p,p}$. We then have isomorphisms of $\mathrm{U}(3)$-modules such as
$$
\Lambda^1 = [\![\Lambda^{1,0}]\!], \quad \Lambda^2 = [\![\Lambda^{2,0}]\!] \oplus [\Lambda^{1,1}], \quad \Lambda^3 = [\![\Lambda^{3,0}]\!] \oplus [\![\Lambda^{2,1}]\!], \text{ etc.}
$$
Each real form of type $(p,q)+(q,p)$ satisfies a specific relation with $J$. To show this, we specialise to the cases $k=2$ and $k=3$.

At every point of $M$, the metric $g$ yields a canonical isomorphism $\mathfrak{so}(6)=\Lambda^2,$ which is obtained by mapping each $A$ in $\mathfrak{so}(6)$ to the two-form $g(A{}\cdot{},{}\cdot{})$. Viewing $\mathfrak{so}(6)$ as the adjoint representation of $\mathrm{SO}(6) \supset \mathrm{U}(3)$, we have actually got an isomorphism of $\mathrm{U}(3)$-modules: for $A \in \mathfrak{so}(6)$ and $B \in \mathrm{U}(3)$, the action of $B$ on two-forms gives
$$Bg(A{}\cdot{},{}\cdot{}) = g(AB^{-1}{}\cdot{},B^{-1}{}\cdot{}) = g(BAB^{-1}{}\cdot{},{}\cdot{}),$$
so the map $A \mapsto g(A{}\cdot{},{}\cdot{})$ is $\mathrm{U}(3)$-equivariant and our claim follows. Now consider the splitting $\mathfrak{so}(6) = \mathfrak{u}(3) \oplus \mathfrak{u}(3)^{\perp}$, where $\mathfrak{u}(3)^{\perp}$ is the orthogonal complement of $\mathfrak{u}(3)$ as a subspace of $\mathfrak{so}(6)$. Any endomorphism $A$ in $\mathfrak{u}(3)$ corresponds to a two-form $\alpha = g(A{}\cdot{},{}\cdot{})$ such that $\alpha(JX,JY) = \alpha(X,Y)$: since $A$ and $J$ commute
\begin{equation}
\label{two_form_properties_1}
\alpha(JX,JY) = g(AJX,JY) = g(JAX,JY) = g(AX,Y) = \alpha(X,Y).
\end{equation} 
On the other hand, a two-form $\beta$ in $[\Lambda^{1,1}]$ is defined so as to vanish on pairs of complex vectors of the same type, namely $\beta(X-iJX,Y-iJY) = 0$. Thus $\beta(JX,JY) = \beta(X,Y)$, and by counting dimensions the following splittings are equivalent:
\begin{equation*}
\mathfrak{so}(6) = \mathfrak{u}(3) \oplus \mathfrak{u}(3)^{\perp}, \qquad \Lambda^2 = [\Lambda^{1,1}] \oplus [\![\Lambda^{2,0}]\!].
\end{equation*}
We have already encountered a two-form enjoying the property of elements in $[\Lambda^{1,1}]$, that is the fundamental two-form $\sigma$. Identity \eqref{two_form_properties_1} is readily checked: 
\begin{equation*}
\sigma(JX,JY) = -g(X,JY) = g(JX,Y) = \sigma(X,Y).
\end{equation*}
Elements of $[\Lambda^{1,1}]$ are then eigenvectors of $J$ with eigenvalue~$+1$. Likewise, elements of $[\![\Lambda^{2,0}]\!]$ are eigenvectors of $J$ with eigenvalue~$-1$, which is now trivial to check. 

Identity \eqref{moving_j_around_eq} implies
\begin{equation}
\label{nabla_sigma_two_zero}
\nabla\sigma(U,JV,JZ) = \nabla\sigma(U,V,J^2Z)= -\nabla\sigma(U,V,Z).
\end{equation}
which we may then rephrase by saying $\nabla \sigma$ sits inside $\Lambda^1 \otimes [\![\Lambda^{2,0}]\!]$. Further, since $M$ is nearly K\"ahler, $\nabla \sigma$ actually takes values in $[\![\Lambda^{3,0}]\!]$. Recall that $\nabla\sigma$ is skew-symmetric, so \eqref{moving_j_around_eq} and \eqref{nabla_sigma_two_zero} imply
\begin{equation*}
\nabla\sigma(JU,JV,Z)+\nabla\sigma(U,JV,JZ)+\nabla\sigma(JU,V,JZ) = -3\nabla\sigma(U,V,Z).
\end{equation*}
On the other hand this is a characteristic property of elements in $[\![\Lambda^{3,0}]\!]$: given $\beta \in \Lambda^3$, this lies in $[\![\Lambda^{3,0}]\!]$ if and only if $J\beta = -3\beta$, where $J$ acts on $\beta$ via
$$J\beta(X,Y,Z) \coloneqq \beta(JX,JY,Z)+\beta(X,JY,JZ)+\beta(JX,Y,JZ),$$
whence $\nabla\sigma \in [\![\Lambda^{3,0}]\!]$.

A last observation is motivated by Remark \ref{remark_1}. Since $\nabla\sigma$ is a three-form, there is a relation between $d\sigma$ and $\nabla \sigma$. This is readily worked out, as $d\sigma = \mathcal{A}\nabla\sigma$, where $(\mathcal{A}\nabla\sigma) (X,Y,Z) \coloneqq \cyclicsum_{ X,Y,Z} \nabla\sigma(X,Y,Z)$, and $\nabla \sigma$ skew-symmetric implies $d\sigma = 3\nabla \sigma$. Conversely, if $d\sigma = 3\nabla \sigma$, then Remark \ref{remark_1} shows $M$ is nearly K\"ahler. We summarise all our observations in
\begin{prop}
\label{first_equivalent_characterisations}
Assume $(M,g,J)$ is an almost Hermitian six-manifold and let $\sigma = g(J{}\cdot{},{}\cdot{})$ be the fundamental two-form. Then the following are equivalent:
\begin{enumerate}
\item $M$ is nearly K\"ahler.
\item $\nabla \sigma$ is a three-form lying in $[\![\Lambda^{3,0}]\!]$.
\item $d\sigma = 3\nabla \sigma$.
\end{enumerate}
\end{prop}

\section{Curvature identities}
\label{curvature}

By Remark \ref{remark_1}, each tangent space of $M$ splits as the orthogonal direct sum of three $J$-invariant planes $$\langle X,JX \rangle \oplus \langle Y,JY \rangle \oplus \langle (\nabla_XJ)Y, J(\nabla_XJ)Y\rangle,$$ where angular brackets denote the real vector space spanned by a pair of vectors and $Y$ is orthogonal to the span of $X$ and $JX$. The following technical lemma states the existence of a special function on~$M$ relating the norm of $(\nabla_XJ)Y$ with $\lVert X\rVert$, $\lVert Y\rVert$, $g(X,Y)$, and $\sigma(X,Y)$.
\begin{lemma}
\label{definition_mu}
There exists a non-negative function $\mu$ on $M$ such that 
\begin{equation}
\label{mu_local_to_global}
\lVert (\nabla_XJ)Y\rVert^2 = \mu^2\bigl(\lVert X\rVert^2 \lVert Y\rVert^2-g(X,Y)^2-\sigma(X,Y)^2\bigr)
\end{equation}
for every pair of vector fields $X,Y$ on $M$.
\end{lemma}
\begin{proof}
We define $\mu$ in terms of a local frame, then we extend it to a global function. Given $X$ and $Y$ in a neighbourhood of a point there exists an orthonormal frame $\{E_i,JE_i\},i=1,2,3$, such that $X = aE_1$ and $Y = bE_1+cJE_1+dE_2$ for local functions $a,b,c,d$. Define $\mu$ by $(\nabla_{E_1}J)E_2 \eqqcolon \mu E_3$. We may assume $\mu$ non-negative up to changing the orientation of the basis. Then $(\nabla_XJ)Y = a(\nabla_{E_1}J)dE_2 = ad\mu E_3$, which implies $\lVert (\nabla_XJ)Y\rVert^2 = a^2d^2\mu^2$. On the other hand, $\lVert X\rVert^2 \lVert Y\rVert^2 = a^2(b^2+c^2+d^2)$, $g(X,Y)^2 = a^2b^2$, and $\sigma(X,Y)^2 = a^2c^2$. So $\mu^2\bigl(\lVert X\rVert^2 \lVert Y\rVert^2-g(X,Y)^2-\sigma(X,Y)^2\bigr) = a^2d^2\mu^2$, and the formula is proved locally.

We can finally extend $\mu$ to a global function by imposing that \eqref{mu_local_to_global} be satisfied for all pairs of vector fields $X,Y$ on $M$. Notice that $\mu$ is defined by the square root of a non-negative function, thus need not be smooth at this stage.
\end{proof}

We shall now study how $\mu$ is related to the Riemannian and the Ricci tensors on $M$, and finally prove that $\mu$ is \emph{constant}. A first step in this direction is to consider second order covariant derivatives of $\sigma$ and study their symmetries. We follow \cite{gray1976} and \cite{morris} for this part. We will sometimes use the notation $\mathfrak{X}(M)$ for the Lie algebra of vector fields on $M$. The Riemannian curvature tensors of type $(3,1)$ and $(4,0)$ will be denoted by the same letter.

\begin{lemma}
\label{second_order_derivatives_sigma}
Let $R \in \Lambda^2 \otimes \mathfrak{so}(2n)$ be the type $(3,1)$ Riemannian curvature tensor of the Levi-Civita connection on $M$, given by $R(X,Y)Z \coloneqq \nabla_X\nabla_YZ-\nabla_Y\nabla_XZ-\nabla_{[X,Y]}Z$. The following identities hold for every quadruple of vector fields $W,X,Y,Z$ on $M$:
\begin{enumerate}
\item $\nabla^2 \sigma(W,X,Y,Z) - \nabla^2 \sigma(X,W,Y,Z) = \sigma(R(X,W)Y,Z)+\sigma(Y,R(X,W)Z).$
\item $\nabla^2 \sigma(X,X,JY,Y)=\lVert (\nabla_XJ)Y\rVert^2$.
\end{enumerate}
\end{lemma}
\begin{proof}
To prove the first formula we show $\nabla^2 \sigma(W,X,Y,Z) - \nabla^2 \sigma(X,W,Y,Z) = g((R(W,X)J)Y,Z)$, then conclude. By expanding the first term
\begin{align*}
\nabla^2 \sigma(W,X,Y,Z) & = W(\nabla \sigma(X,Y,Z)) - \nabla \sigma(\nabla_W X,Y,Z) \\
& \qquad - \nabla \sigma(X,\nabla_WY,Z) - \nabla \sigma(X,Y,\nabla_W Z) \\
& = g(\nabla_W((\nabla_X J)Y),Z)+g((\nabla_XJ)Y,\nabla_WZ) - g((\nabla_{\nabla_WX}J)Y,Z) \\
& \qquad - g((\nabla_XJ)\nabla_WY,Z) - g((\nabla_XJ)Y,\nabla_WZ) \\
& = g((\nabla_W(\nabla_X J))Y,Z)-g((\nabla_{\nabla_WX}J)Y,Z).
\end{align*}
As an element of $\mathfrak{so}(2n)$, $R(W,X)$ is a skew-adjoint derivation. We can then rewrite the difference $\nabla^2 \sigma(W,X,Y,Z) - \nabla^2 \sigma(X,W,Y,Z)$ as 
\begin{align*}
g((R(W,X)J)Y,Z) & = g(R(W,X)JY,Z)-g(JR(W,X)Y,Z) \\
& = -g(JY,R(W,X)Z)-g(JR(W,X)Y,Z) \\
& = \sigma(Y,R(X,W)Z)+\sigma(R(X,W)Y,Z).
\end{align*}
In order to prove the second formula we make use of \eqref{moving_j_around_eq}:
\begin{align*}
\nabla^2 \sigma(X,X,JY,Y) & = X(\nabla\sigma(X,JY,Y))-\nabla\sigma(\nabla_XX,JY,Y) \\
& \qquad - \nabla\sigma(X,\nabla_X JY,Y) - \nabla\sigma(X,JY,\nabla_XY) \\
& = X(\nabla\sigma(JX,Y,Y))-\nabla\sigma(J\nabla_XX,Y,Y) \\
& \qquad - g((\nabla_XJ)\nabla_XJY,Y)-g((\nabla_XJ)JY,\nabla_XY) \\
& = g(\nabla_X JY,(\nabla_XJ)Y)-g((\nabla_XJ)JY,\nabla_XY) \\
& = g((\nabla_XJ)Y,(\nabla_XJ)Y)+g(J\nabla_XY,(\nabla_XJ)Y) \\
& \qquad -g((\nabla_XJ)JY,\nabla_XY) \\
& = \lVert (\nabla_XJ)Y\rVert^2,
\end{align*}
and the statement is proved.
\end{proof}

\begin{lemma}
\label{important_corollary}
Let $R \in \mathrm{Sym}^2(\Lambda^2)$ be the Riemannian curvature $(4,0)$-tensor obtained by contraction with the metric: $R(W,X,Y,Z) \coloneqq g(R(W,X)Y,Z)$. Then
\begin{equation}
\label{norm_nabla_j}
\lVert (\nabla_XJ)Y\rVert^2 = R(X,Y,JX,JY)-R(X,Y,X,Y), \quad X,Y \in \mathfrak{X}(M).
\end{equation}
\end{lemma}
\begin{proof}
Since $\nabla\sigma$ is a three-form, $\nabla^2\sigma(A,B,B,C)=0$. We can then combine Lemma \ref{covariant_derivative_sigma_etc} and the results found in Lemma \ref{second_order_derivatives_sigma} to get
\begin{align*}
\lVert (\nabla_X J)Y \rVert^2 & = \lVert (\nabla_X J)JY \rVert^2 = -\nabla^2 \sigma(X,X,Y,JY) \\
& = \nabla^2 \sigma(X,Y,X,JY)-\nabla^2\sigma(Y,X,X,JY) \\
& = \sigma(R(Y,X)X,JY)+\sigma(X,R(Y,X)JY) \\
& = g(R(Y,X)X,Y)-g(R(Y,X)JX,JY) \\
& = R(X,Y,JX,JY)-R(X,Y,X,Y),
\end{align*}
which was our claim.
\end{proof}

\begin{remark}
Formula \eqref{norm_nabla_j} gives a way to calculate the norm of $(\nabla_XJ)Y$---hence the function $\mu$ in \nolinebreak \eqref{mu_local_to_global}---in terms of the curvature tensor. A remarkable consequence of it is that $R$ is invariant under the action of $J$. To see this, define the tensor $S(W,X,Y,Z) \coloneqq R(JW,JX,JY,JZ)$. Of course $S$ inherites the properties of algebraic curvature tensors, namely $S \in \Lambda^2 \otimes \Lambda^2$ satisfies the first Bianchi identity.
To show $R = S$ we can then check $R(X,Y,Y,X) = S(X,Y,Y,X)$. By formula \eqref{covariant_derivative_j} we have $\lVert (\nabla_{JX}J)JY\rVert = \lVert (\nabla_XJ)Y\rVert$. A straightforward calculation then proves the claim:
\begin{align*}
R(JX,JY,JY,JX) - R(X,Y,Y,X)& = R(JX,JY,JY,JX)-R(X,Y,JY,JX)\\
& \qquad +R(X,Y,JY,JX)-R(X,Y,Y,X) \\
& =  \lVert (\nabla_{JX}J)JY\rVert^2 -\lVert (\nabla_XJ)Y\rVert^2= 0.
\end{align*}
\end{remark}
The identity just obtained allows us to carry out a polarisation process giving a way to measure inner products of vectors of the form $(\nabla_X J)Y$ in terms of the curvature. 
We work out all the details of the next essential result.
\begin{lemma}
For every quadruple of vector fields $W,X,Y,Z$ on $M$ we have the formula
\begin{equation}
\label{inner_products}
g((\nabla_WJ)X,(\nabla_YJ)Z) = R(W,X,JY,JZ)-R(W,X,Y,Z).
\end{equation}
\end{lemma}
\begin{proof}
Mapping $X \mapsto A+B$ in formula \eqref{norm_nabla_j} one has
\begin{align*}
\lVert (\nabla_{A+B}J)Y\rVert^2 & = R(A+B,Y,JA+JB,JY)-R(A+B,Y,A+B,Y) \\
& = R(A,Y,JA,JY)-R(A,Y,A,Y) + R(B,Y,JB,JY)-R(B,Y,B,Y) \\
& \qquad + R(A,Y,JB,JY)-R(A,Y,B,Y)+ R(B,Y,JA,JY)-R(B,Y,A,Y).
\end{align*}
The left hand side is
\begin{equation*}
\lVert (\nabla_{A+B}J)Y\rVert^2 = \lVert (\nabla_A J)Y\rVert^2 + \lVert(\nabla_BJ)Y\rVert^2+2g((\nabla_AJ)Y,(\nabla_BJ)Y),
\end{equation*}
so applying once again \eqref{norm_nabla_j} we find
\begin{align*}
2g((\nabla_AJ)Y,(\nabla_BJ)Y) & = R(A,Y,JB,JY)-R(A,Y,B,Y) \\
& \qquad + R(B,Y,JA,JY)-R(B,Y,A,Y).
\end{align*}
Putting now $Y \mapsto C+D$, we expand $2g((\nabla_AJ)(C+D),(\nabla_BJ)(C+D))$ and obtain the expression
\begin{align*}
& R(A,C,JB,JC)-R(A,C,B,C) + R(A,C,JB,JD)-R(A,C,B,D) \\
& \qquad + R(A,D,JB,JC) - R(A,D,B,C)+ R(A,D,JB,JD) - R(A,D,B,D) \\
& \qquad + R(B,C,JA,JC) - R(B,C,A,C) + R(B,C,JA,JD) - R(B,C,A,D) \\
& \qquad + R(B,D,JA,JC) - R(B,D,A,C)+ R(B,D,JA,JD) - R(B,D,A,D). 
\end{align*}
Linearity in the various arguments implies
\begin{align*}
\MoveEqLeft
2g((\nabla_AJ)(C+D),(\nabla_BJ)(C+D)) \\
& = 2\bigl(g((\nabla_AJ)C,(\nabla_BJ)C)+g((\nabla_AJ)C,(\nabla_BJ)D) \\
& \qquad + g((\nabla_AJ)D,(\nabla_BJ)C) + g((\nabla_AJ)D,(\nabla_BJ)D)\bigr).
\end{align*}
Simplifying we are left with
\begin{align}
\label{step}
\MoveEqLeft
g((\nabla_AJ)C,(\nabla_BJ)D) +g((\nabla_AJ)D,(\nabla_BJ)C ) \nonumber \\
& = R(A,C,JB,JD)-R(A,C,B,D) + R(A,D,JB,JC)-R(A,D,B,C).
\end{align}
Set $L(A,B,C,D) \coloneqq R(A,B,C,D)+R(A,D,C,B)$. The first Bianchi identity, together with \nolinebreak \eqref{step}, gives
\begin{align}
\label{formula_1}
0 & = R(A,B,C,D)+R(B,C,A,D)+R(C,A,B,D) \nonumber\\
& = R(A,B,C,D)-R(C,B,A,D)+\bigl(L(C,A,B,D)-R(C,D,B,A)\bigr) \nonumber\\
& = R(A,B,C,D)+\bigl(L(C,A,B,D)+R(A,B,C,D)\bigr) -\bigl(L(C,B,A,D)-R(C,D,A,B)\bigr) \nonumber\\
& = 3R(A,B,C,D)+L(C,A,B,D)-L(C,B,A,D) \nonumber\\
& = 3R(A,B,C,D)+\bigl(R(C,A,B,D)+R(C,D,B,A)\bigr)-\bigl(R(C,B,A,D)+R(C,D,A,B)\bigr) \nonumber\\
& = 3R(A,B,C,D)-2R(C,D,JA,JB) +R(C,A,JB,JD)-R(C,B,JA,JD) \nonumber \\
& \qquad +2g((\nabla_CJ)D,(\nabla_AJ)B)-g((\nabla_CJ)A,(\nabla_BJ)D)+g((\nabla_CJ)B,(\nabla_AJ)D).
\end{align}
Now we set $C \mapsto JC, D \mapsto JD$:
\begin{align*}
0 & = 3R(A,B,JC,JD)-2R(JC,JD,JA,JB) -R(JC,A,JB,D)+R(JC,B,JA,D) \\
& \qquad +2g((\nabla_{JC}J)JD,(\nabla_AJ)B)-g((\nabla_{JC}J)A,(\nabla_BJ)JD) +g((\nabla_{JC}J)B,(\nabla_AJ)JD).
\end{align*}
Using that $R$ is $J$-invariant, the difference between the latter and \eqref{formula_1} becomes
\begin{align*}
0 & = 3R(A,B,JC,JD)-2R(JC,JD,JA,JB) -R(JC,A,JB,D)+R(JC,B,JA,D) \\
& \qquad +2g((\nabla_{JC}J)JD,(\nabla_AJ)B)-g((\nabla_{JC}J)A,(\nabla_BJ)JD) +g((\nabla_{JC}J)B,(\nabla_AJ)JD) \\
& \qquad - 3R(A,B,C,D)+2R(C,D,JA,JB) -R(C,A,JB,JD)+R(C,B,JA,JD) \\
& \qquad -2g((\nabla_CJ)D,(\nabla_AJ)B)+g((\nabla_CJ)A,(\nabla_BJ)D)-g((\nabla_CJ)B,(\nabla_AJ)D) \\[4pt]
& = 5R(A,B,JC,JD)-5R(A,B,C,D) -R(A,C,JD,JB)-R(A,JC,D,JB)\\
& \qquad -R(A,D,JB,JC)-R(A,JD,JB,C)-4g((\nabla_AJ)B,(\nabla_CJ)D).
\end{align*}
Applying the first Bianchi identity once again we have
\begin{align}
\label{intermediate_step}
4g((\nabla_AJ)B,(\nabla_CJ)D) & = 5R(A,B,JC,JD)-5R(A,B,C,D) \nonumber \\
& \qquad + R(A,JB,C,JD)+R(A,JB,JC,D).
\end{align}
Now map $B \mapsto JB, C \mapsto JC$ and add a fifth of the result to \eqref{intermediate_step}: 
\begin{align*}
\tfrac{24}{5}g((\nabla_AJ)JB,(\nabla_{JC}J)D) & = -R(A,JB,C,JD)-R(A,JB,JC,D) \\
& \qquad -\tfrac{1}{5}R(A,B,JC,JD)+\tfrac{1}{5}R(A,B,C,D) \\
& \qquad +5R(A,B,JC,JD)-5R(A,B,C,D) \\
& \qquad + R(A,JB,C,JD)+R(A,JB,JC,D) \\
& = \tfrac{24}{5}R(A,B,JC,JD)-\tfrac{24}{5}R(A,B,C,D).
\end{align*}
By identity \eqref{covariant_derivative_j} one has $g((\nabla_AJ)JB,(\nabla_{JC}J)D) = g((\nabla_AJ)B,(\nabla_CJ)D)$, and we are done.
\end{proof}

\begin{lemma}
Let $W,X,Y,Z \in \mathfrak{X}(M)$. The following formula holds:
\begin{equation}
\label{cyclic_sum_sigma}
2\nabla^2\sigma(W,X,Y,Z) = -\cyclicsum_{ X,Y,Z} g((\nabla_WJ)X,(\nabla_YJ)JZ).
\end{equation}
\end{lemma}
\begin{proof}
Combine the first formula in Lemma \ref{second_order_derivatives_sigma} and identity \eqref{inner_products}:
\begin{align}
\label{difference_covariant_derivatives_sigma_one}
\nabla^2\sigma(W,X,Y,Z)-\nabla^2\sigma(X,W,Y,Z) & = \sigma(R(X,W)Y,Z)+\sigma(Y,R(X,W)Z) \nonumber \\
& = g(JR(X,W)Y,Z)+g(JY,R(X,W)Z) \nonumber \\
& = g(JR(X,W)Y,Z)-g(R(X,W)JY,Z) \nonumber \\
& = R(X,W,JY,J^2Z)-R(X,W,Y,JZ) \nonumber \\
& = g((\nabla_XJ)W,(\nabla_YJ)JZ).
\end{align}
On the other hand, using \eqref{difference_covariant_derivatives_sigma_one} 
\begin{align*}
\nabla^2\sigma(W,W,Y,Z) & = -\nabla^2\sigma(W,Y,W,Z) \\
& = \nabla^2\sigma(Y,W,W,Z)-\nabla^2\sigma(W,Y,W,Z) \\
& = g((\nabla_WJ)Y,(\nabla_WJ)JZ),
\end{align*}
Polarising the latter, one obtains
\begin{align*}
\nabla^2\sigma(W+X,W+X,Y,Z) & = \nabla^2\sigma(W,W,Y,Z)+\nabla^2\sigma(W,X,Y,Z)\\
& \qquad +\nabla^2\sigma(X,W,Y,Z)+\nabla^2\sigma(X,X,Y,Z) \\
& = g((\nabla_WJ)Y,(\nabla_WJ)JZ)+g((\nabla_XJ)Y,(\nabla_XJ)JZ) \\
& \qquad + \nabla^2\sigma(W,X,Y,Z)+\nabla^2\sigma(X,W,Y,Z),
\end{align*}
whence
\begin{align}
\label{difference_covariant_derivatives_sigma_two}
\MoveEqLeft
\nabla^2\sigma(W,X,Y,Z)+\nabla^2\sigma(X,W,Y,Z) \nonumber \\
& = -g((\nabla_WJ)Y,(\nabla_WJ)JZ)-g((\nabla_XJ)Y,(\nabla_XJ)JZ) +g((\nabla_{W+X}J)Y,(\nabla_{W+X}J)JZ) \nonumber \\
& = g((\nabla_WJ)Y,(\nabla_XJ)JZ)+g((\nabla_XJ)Y,(\nabla_WJ)JZ).
\end{align}
Adding \eqref{difference_covariant_derivatives_sigma_one} to \eqref{difference_covariant_derivatives_sigma_two} and using usual symmetries of $\nabla J$ the claim follows.
\end{proof}

Now we define the Ricci and the Ricci-${*}$ endomorphisms. We still work in dimension $2n$, switching to dimension six in Proposition \ref{mu_constant}.
\begin{definition}
\label{ricci_and_ricci_star}
Given any local, orthonormal frame $E_1,\dots,E_{2n}$, the Ricci and the Ricci-${*}$ endomorphisms $\Ric, \Ricc \in \Lambda^1 \otimes  \Lambda^1$ are given by 
\begin{equation*}
g(\Ric X,Y) \coloneqq \sum_{ i=1}^{ 2n} R(X,E_i,E_i,Y), \quad g(\Ricc X,Y) \coloneqq \sum_{ i=1}^{2n} R(X,E_i,JE_i,JY).
\end{equation*}
\end{definition}
Because of \eqref{inner_products} we can write their difference as
\begin{equation}
\label{difference_ricci}
g((\Ric-\Ricc)X,Y) = \sum_{ i=1}^{2n}g((\nabla_XJ)E_i,(\nabla_YJ)E_i).
\end{equation}
Obviously $\Ric-\Ricc$ is self-adjoint, and so is its covariant derivative. Moreover, $\Ric-\Ricc$ and $J$ commute: set $A \coloneqq \Ric-\Ricc$ and apply formula \eqref{covariant_derivative_j}, so that
\begin{align*}
g(JAX,Y) & = -g(AX,JY) \\
& = -\sum_{ i} g((\nabla_XJ)E_i,(\nabla_{JY}J)E_i) = -\sum_{ i} g(J(\nabla_XJ)E_i,(\nabla_YJ)E_i) \\
& = \sum_{ i} g((\nabla_{JX}J)E_i,(\nabla_YJ)E_i) = g(AJX,Y).
\end{align*}
We can then prove a last useful result.
\begin{lemma}
For $X,Y,Z \in \mathfrak{X}(M)$ we have the following formula:
\begin{align}
\label{nabla_curvature}
2g((\nabla_Z (\Ric-\Ricc))X,Y) & = g((\Ric-\Ricc)JX,(\nabla_ZJ)Y) \nonumber\\
&\qquad +g((\Ric-\Ricc)JY,(\nabla_ZJ)X).
\end{align}
\end{lemma}
\begin{proof}
Start differentiating \eqref{difference_ricci} with $X=Y$, still with $A \coloneqq \Ric-\Ricc$:
\begin{align*}
g((\nabla_ZA)X,X)+2g(AX,\nabla_ZX) & = Z(g(AX,X)) \\
& = 2\sum_{ i=1}^{2n} g(\nabla_Z((\nabla_XJ)E_i),(\nabla_XJ)E_i).
\end{align*}
Rearranging the terms 
\begin{align}
\label{nabla_ric}
g((\nabla_Z A)X,X) = 2\sum_{ i=1}^{2n} g(\nabla_Z((\nabla_XJ)E_i),(\nabla_XJ)E_i)-g((\nabla_XJ)E_i,(\nabla_{\nabla_ZX}J)E_i).
\end{align}
Note that $\sum_{ i=1}^{2n}g((\nabla_XJ)\nabla_Z E_i,(\nabla_XJ)E_i)=0$: setting $\nabla_Z E_i = \sum_{ j=1}^{2n} B_i^jE_j$ we have 
\begin{equation*}
0 = Z(g(E_i,E_j)) = g(\nabla_ZE_i,E_j)+g(E_i,\nabla_ZE_j) = \sum_{ k} B_i^k \delta_{kj}+\sum_{ r} B_j^r\delta_{ir}= B_i^j+B_j^i.
\end{equation*}
Thus 
\begin{align*}
\sum_i g((\nabla_XJ)\nabla_Z E_i,(\nabla_XJ)E_i) & = \sum_{i,j} g((\nabla_XJ)B_i^jE_j,(\nabla_XJ)E_i) \\
& = -\sum_{i,j} g((\nabla_XJ)E_j,(\nabla_XJ)B_j^iE_i)\\
& = -\sum_j g((\nabla_XJ)E_j,(\nabla_XJ)\nabla_ZE_j)=0.
\end{align*}
This last term appears in the expansion of $\nabla^2 \sigma(Z,X,(\nabla_XJ)E_i,E_i)$ as well. Simplifying we get
\begin{align*}
\nabla^2 \sigma(Z,X,(\nabla_XJ)E_i,E_i) & = -Z(g((\nabla_XJ)E_i,(\nabla_XJ)E_i))-g((\nabla_{\nabla_ZX}J)(\nabla_XJ)E_i,E_i) \\
& \qquad -g((\nabla_XJ)\nabla_Z((\nabla_XJ)E_i),E_i)-g((\nabla_XJ)(\nabla_XJ)E_i,\nabla_ZE_i) \\[4pt]
& = -2g(\nabla_Z((\nabla_XJ)E_i),(\nabla_XJ)E_i)+g((\nabla_XJ)E_i,(\nabla_{\nabla_ZX}J)E_i)\\
& \qquad +g(\nabla_Z((\nabla_XJ)E_i),(\nabla_XJ)E_i)-g((\nabla_XJ)(\nabla_XJ)E_i,\nabla_ZE_i) \\[4pt]
& = g((\nabla_XJ)E_i,(\nabla_{\nabla_Z X}J)E_i)-g((\nabla_Z(\nabla_XJ))E_i,(\nabla_XJ)E_i) \\
& \qquad +g((\nabla_XJ)E_i,(\nabla_XJ)\nabla_ZE_i).
\end{align*}
Therefore, by formula \eqref{cyclic_sum_sigma}, identity \eqref{nabla_ric} becomes (all sums are over $i=1,\dots,2n$)
\begin{align*}
\MoveEqLeft 
g((\nabla_Z (\Ric-\Ricc))X,X) \\
& = 2\sum g(\nabla_Z((\nabla_XJ)E_i),(\nabla_XJ)E_i)-g((\nabla_XJ)E_i,(\nabla_{\nabla_ZX}J)E_i) \\[4pt]
& = - 2\sum \nabla^2\sigma(Z,X,(\nabla_XJ)E_i,E_i) \\[4pt]
& = \sum g((\nabla_ZJ)X,(\nabla_{(\nabla_XJ)E_i}J)JE_i) + g((\nabla_ZJ)(\nabla_XJ)E_i,(\nabla_{E_i}J)JX) \\
& \qquad + g((\nabla_ZJ)E_i,(\nabla_XJ)J(\nabla_XJ)E_i) \\[4pt]
& = \sum g((\nabla_{E_i}J)(\nabla_ZJ)X,J(\nabla_XJ)E_i)+ g((\nabla_ZJ)(\nabla_XJ)E_i,J(\nabla_XJ)E_i)\\
& \qquad +g((\nabla_XJ)(\nabla_ZJ)E_i,(\nabla_XJ)JE_i).
\end{align*}
The second term in the latter sum vanishes by \eqref{moving_j_around_eq}. The sum $\sum g((\nabla_XJ)(\nabla_ZJ)E_i,(\nabla_XJ)JE_i)$ vanishes as well. To see this, we set $C \coloneqq J(\nabla_ZJ)$. In the first place $C$ lies in $\mathfrak{so}(2n)$, because
\begin{equation*}
g(J(\nabla_ZJ)E_i,E_j) = g((\nabla_ZJ)JE_j,E_i) = -g(J(\nabla_ZJ)E_j,E_i).
\end{equation*}
Consequently, the following chain of identities leads to our claim (indices $i,j$ vary from $1$ to $2n$):
\begin{align*}
\sum g((\nabla_XJ)(\nabla_ZJ)E_i,(\nabla_XJ)JE_i) &= -\sum g((\nabla_XJ)J(\nabla_ZJ)E_i,(\nabla_XJ)E_i) \\
& = -\sum g((\nabla_XJ)C_i^jE_j,(\nabla_XJ)E_i) \\
& = \sum g((\nabla_XJ)E_j,(\nabla_XJ)C_j^iE_i) \\
& = \sum g((\nabla_XJ)J(\nabla_ZJ)E_j,(\nabla_XJ)E_j) = 0.
\end{align*}
We then go back to our first expansion recalling that $\Ric-\Ricc$ commutes with $J$.
\begin{align*}
-2\sum \nabla^2\sigma(Z,X,(\nabla_XJ)E_i,E_i) & = \sum g((\nabla_{E_i}J)(\nabla_ZJ)X,J(\nabla_XJ)E_i) \\ 
& = -\sum g((\nabla_{J(\nabla_ZJ)X}J)E_i,(\nabla_XJ)E_i) \\
& = -g((\Ric-\Ricc)J(\nabla_ZJ)X,X) \\
& = g((\Ric-\Ricc)JX,(\nabla_ZJ)X).
\end{align*}
Thus $g((\nabla_Z (\Ric-\Ricc))X,X) = g((\Ric-\Ricc)JX,(\nabla_ZJ)X)$. By polarisation and the symmetry of $\nabla_Z (\Ric-\Ricc)$ the result follows.
\end{proof}

Let us restrict to the six-dimensional case now, so $n=3$. Recall that in Lemma \ref{definition_mu} we proved the existence of a special function $\mu$ on $M$ satisfying \eqref{mu_local_to_global}. 
\begin{prop}
\label{mu_constant}
If $M$ is a nearly K\"ahler six-manifold, the function $\mu$ is constant.
\end{prop}
\begin{proof}
We only prove $\mu$ is locally constant, then the claim follows from the connectedness of $M$. Mapping $X$ into $A+B$ in \eqref{mu_local_to_global} one has
\begin{equation*}
g((\nabla_{A+B}J)Y,(\nabla_{A+B}J)Y) = \mu^2\bigl(\lVert A+B\rVert^2\lVert Y\rVert^2-g(A+B,Y)^2-\sigma(A+B,Y)^2\bigr),
\end{equation*}
which can be simplified as
\begin{align*}
g((\nabla_AJ)Y,(\nabla_BJ)Y) & = \mu^2\bigl(g(A,B)\lVert Y\rVert^2 - g(A,Y)g(B,Y)-g(JA,Y)g(JB,Y)\bigr).
\end{align*}
On the other hand, using a local $\mathrm{U}(3)$-adapted frame $\{E_i,JE_i\}, i=1,2,3$, we can write
\begin{align*}
g((\Ric-\Ricc)A,B) & = \sum_{ i=1}^{ 3} g((\nabla_AJ)E_i,(\nabla_BJ)E_i)+g((\nabla_AJ)JE_i,(\nabla_BJ)JE_i) \\
& = \mu^2\bigl(6g(A,B)-g(A,B)-g(JA,JB)\bigr)=4\mu^2g(A,B).
\end{align*}
Thus 
\begin{equation}
\label{diff_ricci}
\Ric-\Ricc = 4\mu^2 \id,
\end{equation}
but now formula \eqref{nabla_curvature} implies
\begin{align*}
2g((\nabla_Z(\Ric-\Ricc))X,Y) & = g((\Ric-\Ricc)JX,(\nabla_ZJ)Y) + g((\Ric-\Ricc)JY,(\nabla_ZJ)X) \\
& = 4\mu^2\bigl(g(JX,(\nabla_ZJ)Y)+g(JY,(\nabla_ZJ)X)\bigr)=0.
\end{align*}
This proves $\nabla_Z (\Ric-\Ricc)= 0 = 4Z(\mu^2)\id$ for every $Z$, hence $\mu$ is locally constant.
\end{proof}
We have thus proved that on connected nearly K\"ahler six-manifolds there exists a constant $\mu$ such that
\begin{equation*}
\lVert (\nabla_X J)Y\rVert^2 = \mu^2\bigl(\lVert X\rVert^2 \lVert Y\rVert^2 - g(X,Y)^2-\sigma(X,Y)^2 \bigr), \quad X,Y \in \mathfrak{X}(M).
\end{equation*}
Observe $\mu$ cannot vanish because of the nearly K\"ahler condition, so we assume it to be positive according to Lemma \ref{definition_mu}. Using the terminology introduced by Gray \cite[Proposition 3.5]{gray} we say that connected nearly K\"ahler six-manifolds have \emph{global constant type}. 

\section{The Einstein condition}
\label{the_einstein_condition}

The aim of this section is to push our calculations further in order to prove that nearly K\"ahler six-manifolds are Einstein. We follow \cite{gray1976} to do this. We first introduce a connection adapted to the $\mathrm{U}(3)$-structure $(g,J)$. A quick computation of the torsion of $J$ will help us go smoothly towards it. We then work out some relevant symmetries satisfied by the curvature tensor of the new connection. We conclude proving that $\Ric_g = 5\mu^2g$, where $\Ric_g$ is the Ricci curvature $(2,0)$-tensor of the Levi-Civita connection and $\mu$ is the function defined in \eqref{mu_local_to_global}. 

Let us now compute the Nijenhuis tensor of $J$, i.e.\ the type $(2,1)$-tensor field $N$ on $M$ defined by
\begin{equation*}
4N(X,Y) \coloneqq  [X,Y]-[JX,JY]+J[JX,Y]+J[X,JY], \quad X,Y \in \mathfrak{X}(M).
\end{equation*}
\begin{prop}
\label{nijenhuis_nk}
If $M$ is nearly K\"ahler then $N(X,Y) = J(\nabla_X J)Y$, where $X,Y \in \mathfrak{X}(M)$.
\end{prop}

\begin{proof}
The key property we use here is that the Levi-Civita connection $\nabla$ is torsion-free. Expanding the commutators one gets
\begin{align*}
4N(X,Y) & = \nabla_X Y - \nabla_Y X + 2J(\nabla_XJ)Y+J\nabla_X JY - J\nabla_Y JX \\
& = 2J(\nabla_XJ)Y+J(\nabla_X J)Y - J(\nabla_Y J)X\\
& = 4J(\nabla_X J)Y,
\end{align*}
and we are done.
\end{proof}
The difference $\nabla - \tfrac{1}{2}N$ defines a covariant derivative $\widehat{\nabla}$:
\begin{equation*}
\widehat{\nabla}_XY \coloneqq \nabla_XY-\tfrac{1}{2}J(\nabla_XJ)Y,  \quad X,Y \in \mathfrak{X}(M).
\end{equation*}
\begin{prop}
\label{unitary_connection}
$\widehat{\nabla}$ is a $\mathrm{U}(n)$-connection.
\end{prop}
\begin{remark}
In Proposition \ref{su3_connection} below we prove that on nearly K\"ahler six-manifolds $\widehat{\nabla}$ is actually an $\mathrm{SU}(3)$-connection, first exhibiting a complex volume form $\psi_{\mathbb{C}}$ on $M$ and then proving it is $\widehat{\nabla}$-parallel.
\end{remark}
\begin{proof}
It is enough to show $\widehat{\nabla} g = 0$ and $\widehat{\nabla} J = 0$. Notice that $J(\nabla_XJ)$ is skew-adjoint: recall $\nabla\sigma$ is a three-form by Proposition \ref{first_equivalent_characterisations}, and that $J$ and $\nabla J$ anti-commute, so
\begin{align*}
g(J(\nabla_XJ)Y,Z) & = -g((\nabla_XJ)Y,JZ) = -\nabla\sigma(X,Y,JZ) \\
& = \nabla\sigma(X,JZ,Y) = g((\nabla_XJ)JZ,Y) = -g(Y,J(\nabla_XJ)Z).
\end{align*}
Since $\nabla g = 0$, the above identity implies
\begin{align*}
\widehat{\nabla}g(X,Y,Z) & = X(g(Y,Z))-g(\nabla_XY,Z)-g(Y,\nabla_XZ) \\
& \qquad + \tfrac{1}{2}\bigl(g(J(\nabla_XJ)Y,Z)+g(Y,J(\nabla_XJ)Z)\bigr) = 0.
\end{align*}
The second claim follows easily by expanding $(\widehat{\nabla}_XJ)Y = \widehat{\nabla}_X JY - J\widehat{\nabla}_XY$ and simplifying.
\end{proof}

Let us call $\widehat{R}$ the curvature tensor of $\widehat{\nabla}$: $\widehat{R}(W,X)Y \coloneqq \widehat{\nabla}_W\widehat{\nabla}_XY-\widehat{\nabla}_X\widehat{\nabla}_WY - \widehat{\nabla}_{[W,X]}Y$. Standard computations give
\begin{align*}
\widehat{R}(W,X)Y & = R(W,X)Y + \tfrac{1}{4}\bigl((\nabla_XJ)(\nabla_WJ)Y-(\nabla_WJ)(\nabla_XJ)Y\bigr) \\
& \qquad -\tfrac{1}{2}J(R(W,X)JY-JR(W,X)Y).
\end{align*}
A contraction with the metric and identity \eqref{inner_products} applied to the last term yield a type $(4,0)$-tensor field, which we still denote by $\widehat{R}$. Its expression is
\begin{align}
\label{characteristic_curvature}
\widehat{R}(W,X,Y,Z) & = R(W,X,Y,Z)+\tfrac{1}{2}g((\nabla_WJ)X,(\nabla_YJ)Z) \nonumber \\
& \qquad +\tfrac{1}{4}\bigl(g((\nabla_XJ)Y,(\nabla_WJ)Z)-g((\nabla_WJ)Y,(\nabla_XJ)Z) \bigr).
\end{align}
We can go a bit further rewriting every summand in terms of the curvature tensor $R$: by formula \eqref{inner_products} and the first Bianchi identity, \eqref{characteristic_curvature} becomes
\begin{align*}
\widehat{R}(W,X,Y,Z) & = R(W,X,Y,Z)+\tfrac{1}{2}\bigl(R(W,X,JY,JZ)-R(W,X,Y,Z)\bigr) \nonumber \\
& \qquad +\tfrac{1}{4}\bigl(R(X,Y,JW,JZ)-R(X,Y,W,Z)\bigr) \\
& \qquad \qquad -R(W,Y,JX,JZ)+R(W,Y,X,Z)\bigr) \\
& = \tfrac{1}{4}\bigl(3R(W,X,Y,Z)+2R(W,X,JY,JZ) \\
& \qquad +R(X,Y,JW,JZ)-R(W,Y,JX,JZ)\bigr).
\end{align*}
Recalling that $R$ is $J$-invariant and lies in $\mathrm{Sym}^2(\Lambda^2)$ we obtain the final expression
\begin{align}
\label{characteristic_tensor_with_curvature}
\widehat{R}(W,X,Y,Z) & = \tfrac{1}{4}\bigl(3R(W,X,Y,Z)+2R(W,X,JY,JZ) \nonumber \\
& \qquad +R(W,Z,JX,JY)+R(W,Y,JZ,JX) \bigr).
\end{align}
\begin{lemma}
\label{symmetry_of_s_with_j}
The tensor $\widehat{R}$ lies in $\Lambda^2 \otimes [\Lambda^{1,1}]$.
\end{lemma}
\begin{proof}
Skew-symmetry in the first two arguments is straightforward by definition of $\widehat{R}$. That $\widehat{R}(W,X)$ sits in $[\Lambda^{1,1}]$ is a simple consequence of Proposition \ref{unitary_connection}.
\end{proof}
\begin{lemma}
\label{symmetries_of_s}
The tensor $\widehat{R}$ sits inside $\mathrm{Sym}^2([\Lambda^{1,1}])$.
\end{lemma}
\begin{proof}
Lemma \ref{symmetry_of_s_with_j} implies that we only need to check $\widehat{R}(W,X,Y,Z) = \widehat{R}(Y,Z,W,X)$. This can be done using \eqref{characteristic_tensor_with_curvature} and applying $J$-invariance of $R$.
\end{proof}
We now want more information about the exact expression of $\nabla \widehat{R}$. We keep working on a nearly K\"ahler manifold of generic dimension $2n$, focussing on the six-dimensional case only after Proposition~\ref{main_result_for_einstein}. Incidentally, in the course of the proof of that result we will need an explicit formula for the cyclic sum $\nabla_V\widehat{R}(W,X,Y,Z)+\nabla_W\widehat{R}(X,V,Y,Z)+\nabla_X\widehat{R}(V,W,Y,Z)$, specifically the case where $V,W,X$ are elements of a local unitary frame. The goal now is to work out this expression.

Let us start computing $\nabla_V\widehat{R}(W,X,Y,Z)$. Differentiating \eqref{characteristic_curvature} one gets
\begin{align*}
V(\widehat{R}(W,X,Y,Z)) & = V(R(W,X,Y,Z))+\tfrac{1}{4}g(\nabla_V((\nabla_XJ)Y),(\nabla_WJ)Z) \\
& \qquad +\tfrac{1}{4}g((\nabla_XJ)Y,\nabla_V((\nabla_WJ)Z)) -\tfrac{1}{4}g(\nabla_V((\nabla_WJ)Y),(\nabla_XJ)Z)\\
& \qquad -\tfrac{1}{4}g((\nabla_WJ)Y,\nabla_V((\nabla_XJ)Z))+\tfrac{1}{2}g(\nabla_V((\nabla_WJ)X),(\nabla_YJ)Z) \\
& \qquad +\tfrac{1}{2}g((\nabla_WJ)X,\nabla_V((\nabla_YJ)Z)).
\end{align*}
Expanding both sides and isolating $\nabla_V\widehat{R}(W,X,Y,Z)$ on the left we have
\begin{align*}
\MoveEqLeft
\nabla_V\widehat{R}(W,X,Y,Z) \\
& = -\widehat{R}(\nabla_VW,X,Y,Z)\negthinspace-\negthinspace\widehat{R}(W,\nabla_VX,Y,Z)\negthinspace-\negthinspace\widehat{R}(W,X,\nabla_VY,Z) \negthinspace-\negthinspace\widehat{R}(W,X,Y,\nabla_VZ)\\
& \qquad +R(\nabla_VW,X,Y,Z)\negthinspace+\negthinspace R(W,\nabla_VX,Y,Z) \negthinspace+\negthinspace R(W,X,\nabla_VY,Z)\negthinspace+\negthinspace R(W,X,Y,\nabla_VZ) \\[4pt]
& \qquad +\tfrac{1}{4}\bigl(g((\nabla_V(\nabla_XJ))Y+(\nabla_XJ)\nabla_VY,(\nabla_WJ)Z)\\
& \qquad \qquad +g((\nabla_XJ)Y,(\nabla_V(\nabla_WJ))Z+(\nabla_WJ)\nabla_VZ)\bigr) \\[4pt]
& \qquad -\tfrac{1}{4}\bigl(g((\nabla_V(\nabla_WJ))Y+(\nabla_WJ)\nabla_VY,(\nabla_XJ)Z) \\
& \qquad \qquad +g((\nabla_WJ)Y,(\nabla_V(\nabla_XJ))Z+(\nabla_XJ)\nabla_VZ)\bigr) \\[4pt]
& \qquad +\tfrac{1}{2}\bigl(g((\nabla_V(\nabla_WJ))X+(\nabla_WJ)\nabla_VX,(\nabla_YJ)Z) \\
& \qquad \qquad +g((\nabla_WJ)X,(\nabla_V(\nabla_YJ))Z+(\nabla_YJ)\nabla_VZ)\bigr) +\nabla_VR(W,X,Y,Z).
\end{align*}
One can expand the first four summands on the right hand side making use of \eqref{characteristic_curvature}. Recall that $(\nabla_{A,B}^2J)C = (\nabla_A(\nabla_BJ))C-(\nabla_{\nabla_AB}J)C$, then simplifying we are left with
\begin{align*}
\MoveEqLeft
\nabla_V\widehat{R}(W,X,Y,Z) \\
& = \nabla_VR(W,X,Y,Z) \\
& \qquad +\tfrac{1}{2}\bigl(g((\nabla_{V,W}^2J)X,(\nabla_YJ)Z)+g((\nabla_{V,Y}^2J)Z,(\nabla_WJ)X)\bigr) \\
& \qquad +\tfrac{1}{4}\bigl(g((\nabla_{V,W}^2J)Z,(\nabla_XJ)Y)+g((\nabla_{V,X}^2J)Y,(\nabla_WJ)Z)\bigr) \\
& \qquad -\tfrac{1}{4}\bigl(g((\nabla_{V,W}^2J)Y,(\nabla_XJ)Z)+g((\nabla_{V,X}^2J)Z,(\nabla_WJ)Y)\bigr).
\end{align*}
Therefore, the second Bianchi identity implies
\begin{align}
\label{cyclic_sum_last_result}
\MoveEqLeft
\nabla_V\widehat{R}(W,X,Y,Z)+\nabla_W\widehat{R}(X,V,Y,Z)+\nabla_X\widehat{R}(V,W,Y,Z) \nonumber \\
& = \cyclicsum_{V,W,X} \Bigl(\tfrac{1}{2}g((\nabla_{V,Y}^2J)Z,(\nabla_WJ)X)+\tfrac{1}{2}g((\nabla_{V,W}^2J)X,(\nabla_YJ)Z)\nonumber \\[-2mm]
& \qquad \qquad +\tfrac{1}{4}g((\nabla_{V,W}^2J)Z-(\nabla_{W,V}^2J)Z,(\nabla_XJ)Y) \nonumber \\
& \qquad \qquad +\tfrac{1}{4}g((\nabla_{V,X}^2J)Y-(\nabla_{X,V}^2J)Y,(\nabla_WJ)Z)\Bigr).
\end{align}
Besides formula \eqref{cyclic_sum_last_result}, in the proof of Proposition \ref{main_result_for_einstein} we will need a last technical result. 
\begin{lemma}
Let $Y$ be a vector field on $M$ and $\{E_i,JE_i\}_{i=1,\dots,n}$, with $JE_i = E_{n+i}$, be a local orthonormal frame. Then the following formula holds:
\begin{equation}
\label{another_formula_ricci}
\sum_{ j=1}^{ 2n}(\nabla_{E_j,E_j}^2 J)Y = -(\Ric-\Ricc)JY.
\end{equation}
\end{lemma}
\begin{proof}
This is a consequence of formula \eqref{cyclic_sum_sigma}:
\begin{align*}
g((\nabla_{E_j,E_j}^2J)Y,X) & = \nabla^2\sigma(E_j,E_j,Y,X) \\
& = -\tfrac{1}{2}\bigl(g((\nabla_{E_j}J)Y,(\nabla_XJ)JE_j)+g((\nabla_{E_j}J)X,(\nabla_{E_j}J)JY)\bigr) \\
& = \tfrac{1}{2}\bigl(g((\nabla_{E_j}J)Y,J(\nabla_XJ)E_j)-g(J(\nabla_{E_j}J)X,(\nabla_{E_j}J)Y)\bigr) \\ 
& = \tfrac{1}{2}\bigl(g((\nabla_{E_j}J)Y,J(\nabla_XJ)E_j)+g(J(\nabla_XJ)E_j,(\nabla_{E_j}J)Y)\bigr) \\
& = g((\nabla_{E_j}J)Y,(\nabla_{E_j}J)JX).
\end{align*}
Then summing over $j$ and identity \eqref{difference_ricci} give
\begin{align*}
\sum_{ j=1}^{2n} g\bigl((\nabla_{E_j,E_j}^2J)Y,X\bigr) & = \sum_{ j=1}^{ 2n} g((\nabla_{E_j}J)Y,(\nabla_{E_j}J)JX) \\
& = g((\Ric-\Ricc)Y,JX) = -g((\Ric-\Ricc)JY,X),
\end{align*}
because $\Ric-\Ricc$ commutes with $J$.
\end{proof}

\begin{prop}
\label{main_result_for_einstein}
Let $W,X$ be two vector fields on $M$ and $\{E_i,JE_i\}_{i=1,\dots,n}$, be a local orthonormal frame as above. Then
\begin{equation}
\label{main_formula_for_einstein}
\sum_{ i,j=1}^{2n} g((\Ric-\Ricc)E_i,E_j)\bigl(R(W,E_i,E_j,X)-5R(W,E_i,JE_j,JX)\bigr)=0.
\end{equation}
\end{prop}
\begin{proof}
Since $\widehat{R} \in \Lambda^2 \otimes [\Lambda^{1,1}]$ by Lemma \ref{symmetry_of_s_with_j} and $JE_i = E_{n+i}$ for $i=1,\dots,n$, we have
\begin{align*}
\sum_{ i=1}^{2n} \widehat{R}(W,X,E_i,(\nabla_VJ)E_i) & = \tfrac{1}{2}\sum_{ i=1}^{2n} \widehat{R}(W,X,E_i,(\nabla_VJ)E_i)+\widehat{R}(W,X,JE_i,J(\nabla_VJ)E_i) \\
& = \tfrac{1}{2}\sum_{ i=1}^{ 2n} \widehat{R}(W,X,E_i,(\nabla_VJ)E_i)-\widehat{R}(W,X,JE_i,(\nabla_VJ)JE_i) \\
& = \tfrac{1}{2}\sum_{ i=1}^{n} \widehat{R}(W,X,E_i,(\nabla_VJ)E_i)-\widehat{R}(W,X,JE_i,(\nabla_VJ)JE_i) \\
& \qquad + \tfrac{1}{2}\sum_{ i=1}^{ n} \widehat{R}(W,X,JE_i,(\nabla_VJ)JE_i)-\widehat{R}(W,X,E_i,(\nabla_VJ)E_i)=0.
\end{align*}
We can thus differentiate the identity obtained with respect to a vector field $U$ viewing each summand on the left hand side as a function $p \mapsto \widehat{R}_p({}\cdot{},{}\cdot{},{}\cdot{},(\nabla_VJ)_p{}\cdot{})$:
\begin{equation}
\label{differential_s}
\sum_i \nabla_U \widehat{R}(W,X,E_i,(\nabla_VJ)E_i)+\widehat{R}(W,X,E_i,(\nabla_{U,V}^2J)E_i) = 0.
\end{equation}
Set $U=V=E_j$ and sum over $j=1,\dots,2n$. The second term in the latter sum becomes 
\begin{equation}
\label{sum_double_covariant_derivative_j}
\sum_{i,j}\widehat{R}(W,X,E_i,(\nabla_{E_j,E_j}^2J)E_i).
\end{equation}
By \eqref{another_formula_ricci}, sum \eqref{sum_double_covariant_derivative_j} becomes
\begin{align*}
\sum_{i,j}\widehat{R}(W,X,E_i,(\nabla_{E_j,E_j}^2J)E_i) & =-\sum_i\widehat{R}(W,X,E_i,(\Ric-\Ricc)JE_i) \\
& = -\sum_{i,j} \widehat{R}(W,X,E_i,g((\Ric-\Ricc)JE_i,JE_j)JE_j) \\
& = -\sum_{i,j} g((\Ric-\Ricc)E_i,E_j)\widehat{R}(W,X,E_i,JE_j).
\end{align*}
Set $X=JW$. Then $J$-invariance of $R$ and the first Bianchi identity give
\begin{align*}
\widehat{R}(W,JW,E_i,JE_j) & = \tfrac{1}{4}\bigl(3R(W,JW,E_i,JE_j)-2R(W,JW,JE_i,E_j)\\
& \qquad -R(JW,E_i,JW,E_j)-R(W,E_i,W,E_j))\bigr) \\
& = \tfrac{1}{4}\bigl(5R(W,JW,E_i,JE_j)-R(W,E_i,W,E_j)-R(W,JE_i,W,JE_j)\bigr) \\
& = \tfrac{1}{4}\bigl(5R(W,E_i,JW,JE_j)-5R(W,JE_j,JW,E_i) \\
& \qquad -R(W,E_i,W,E_j)-R(W,JE_i,W,JE_j)\bigr).
\end{align*}
Using  \eqref{characteristic_tensor_with_curvature} and \eqref{another_formula_ricci} we have (sums over $i$ and $j$)
\begin{align*}
\MoveEqLeft
\sum \widehat{R}(W,JW,E_i,(\nabla_{E_j,E_j}^2J)E_i) \\
& = -\sum g((\Ric-\Ricc)E_i,E_j)\widehat{R}(W,JW,E_i,JE_j) \\
& = \tfrac{1}{4}\sum g((\Ric-\Ricc)E_i,E_j)\bigl(-5R(W,E_i,JW,JE_j)+5R(W,JE_j,JW,E_i)\\
& \qquad \qquad \qquad \qquad \qquad \qquad \qquad +R(W,E_i,W,E_j)+R(W,JE_i,W,JE_j)\bigr).
\end{align*}
We now split this expression in four different sums where the indices $i,j$ always run from $1$ to $n$. Set $A \coloneqq \Ric-\Ricc$ and
\begin{align*}
L(E_i,E_j) & \coloneqq -5R(W,E_i,JW,JE_j)+5R(W,JE_j,JW,E_i) \\
H(E_i,E_j) & \coloneqq R(W,E_i,W,E_j)+R(W,JE_i,W,JE_j),
\end{align*}
so we can write $\sum_{i,j}\widehat{R}(W,JW,E_i,(\nabla_{E_j,E_j}^2J)E_i)$ as
\begin{align*}
& \tfrac{1}{4}\sum_{i,j=1}^{ n} \bigl(g(AE_i,E_j)(L+H)(E_i,E_j)+ g(AE_i,JE_j)(L+H)(E_i,JE_j) \\
& \qquad +g(AJE_i,E_j)(L+H)(JE_i,E_j) +g(AJE_i,JE_j)(L+H)(JE_i,JE_j)\bigr).
\end{align*}
The symmetries of $R$, its $J$-invariance and the identity $AJ = JA$ yield
\begin{align*}
\MoveEqLeft
\sum \widehat{R}(W,JW,E_i,(\nabla_{E_j,E_j}^2J)E_i) \\
& = \tfrac{1}{2}\sum_{ i,j=1}^{ n} \Bigl(g(AE_i,E_j)\bigl(L(E_j,E_i)+H(E_i,E_j)\bigr)+g(AE_i,JE_j)\bigl(L(E_i,JE_j)+H(E_i,JE_j)\bigr)\Bigr).
\end{align*}
Going back to our usual notation we find
\begin{align*}
\MoveEqLeft
\sum_{ i,j=1}^{2n}\widehat{R}(W,JW,E_i,(\nabla_{E_j,E_j}^2J)E_i) \\[-2pt]
& = \tfrac{1}{2}\sum_{ i,j=1}^{2n} g(AE_i,E_j)\bigl(R(W,E_i,W,E_j)-5R(W,E_i,JW,JE_j)\bigr) \\[-2pt]
& \qquad +\tfrac{1}{2}\sum_{ i,j=1}^{n} g(AE_i,JE_j)\bigl(R(W,E_i,W,JE_j)+5R(W,E_i,JW,E_j)\bigr) \\[-2pt]
& \qquad +\tfrac{1}{2}\sum_{ i,j=1}^{n} g(AJE_i,E_j)\bigl(R(W,JE_i,W,E_j)-5R(W,JE_i,JW,JE_j) \bigr) \\[-2pt]
& \qquad +\tfrac{1}{2}\sum_{ i,j=1}^{n} g(AJE_i,JE_j)\bigl(R(W,JE_i,W,JE_j)+5R(W,JE_i,JW,E_j)\bigr) \\[-2pt]
& = \tfrac{1}{2}\sum_{ i,j=1}^{2n} g(AE_i,E_j)\bigl(R(W,E_i,W,E_j)-5R(W,E_i,JW,JE_j)\bigr).
\end{align*}
Let us go back to \eqref{differential_s} and focus on the first term now. Setting again $U=V=E_j,X=JW$, applying Lemma \ref{symmetries_of_s}, and summing over $j$ (and $k$) from $1$ to $2n$ we have:
\begin{align*}
\sum_{i,j} \nabla_{E_j}\widehat{R}(W,JW,E_i,(\nabla_{E_j}J)E_i) & = \sum_{i,j,k} \nabla_{E_j}\widehat{R}(W,JW,E_i,g((\nabla_{E_j}J)E_i,E_k)E_k) \\
& = \sum_{i,j,k} \nabla\sigma(E_j,E_i,E_k)\nabla_{E_j}\widehat{R}(E_i,E_k,W,JW) \\
& = \tfrac{1}{2}\sum_{i<j<k}\nabla\sigma(E_i,E_j,E_k) \cyclicsum_{ i,j,k} \nabla_{E_i}\widehat{R}(E_j,E_k,W,JW).
\end{align*}
The sum $\cyclicsum_{ i,j,k}\nabla_{E_i}\widehat{R}(E_j,E_k,W,JW)$ actually vanishes: by formula \eqref{cyclic_sum_last_result}
\begin{align*}
\MoveEqLeft
\nabla_{E_i}\widehat{R}(E_j,E_k,W,JW)+\nabla_{E_j}\widehat{R}(E_k,E_i,W,JW)+\nabla_{E_k}\widehat{R}(E_i,E_j,W,JW) \\
& = \tfrac{1}{2}\cyclicsum_{ i,j,k} \bigl( g((\nabla_{E_k,W}^2J)JW,(\nabla_{E_i}J)E_j) + g((\nabla_{E_i,E_j}^2J)E_k,(\nabla_WJ)JW)\bigr) \\
& \qquad +\tfrac{1}{4}\cyclicsum_{ i,j,k} g((\nabla_{E_i,E_j}^2J)JW-(\nabla_{E_j,E_i}^2J)JW,(\nabla_{E_k}J)W)\\
& \qquad +\tfrac{1}{4}\cyclicsum_{ i,j,k} g((\nabla_{E_j,E_i}^2J)W-(\nabla_{E_i,E_j}^2J)W,(\nabla_{E_k}J)JW).
\end{align*}
Recall that $\nabla^2\sigma(W,X,Y,Z) = g((\nabla_{W,X}^2J)Y,Z)$. Applying \eqref{cyclic_sum_sigma} and simplifying we have
\begin{align*}
\MoveEqLeft
\nabla_{E_i}\widehat{R}(E_j,E_k,W,JW)+\nabla_{E_j}\widehat{R}(E_k,E_i,W,JW)+\nabla_{E_k}\widehat{R}(E_i,E_j,W,JW) \\
& = \tfrac{1}{2}\cyclicsum_{ i,j,k} \nabla^2\sigma(E_k,W,JW,(\nabla_{E_i}J)E_j) \\
& \qquad +\tfrac{1}{4}\cyclicsum_{ i,j,k} \bigl(\nabla^2\sigma(E_i,E_j,JW,(\nabla_{E_k}J)W)-\tfrac{1}{4}\nabla^2\sigma(E_j,E_i,JW,(\nabla_{E_k}J)W)\bigr) \\
& \qquad -\tfrac{1}{4}\cyclicsum_{ i,j,k} \bigl(\nabla^2\sigma(E_i,E_j,W,(\nabla_{E_k}J)JW)+\tfrac{1}{4}\nabla^2\sigma(E_j,E_i,W,(\nabla_{E_k}J)JW)\bigr)\\
& = \tfrac{1}{2}g((\nabla_WJ)(\nabla_{E_i}J)E_j,(\nabla_{E_k}J)W) +\tfrac{1}{2}g((\nabla_{E_k}J)E_i,(\nabla_WJ)(\nabla_{E_j}J)W) \\
& \qquad +\tfrac{1}{2}g((\nabla_WJ)(\nabla_{E_j}J)E_k,(\nabla_{E_i}J)W)+\tfrac{1}{2}g((\nabla_{E_i}J)E_j,(\nabla_WJ)(\nabla_{E_k}J)W) \\
& \qquad +\tfrac{1}{2}g((\nabla_{E_j}J)E_k,(\nabla_WJ)(\nabla_{E_i}J)W)+\tfrac{1}{2}g((\nabla_WJ)(\nabla_{E_k}J)E_i,(\nabla_{E_j}J)W) =0.
\end{align*}
Then $\sum_{ i,j=1}^{2n} \nabla_{E_j}\widehat{R}(W,JW,E_i,(\nabla_{E_i}J)E_j)=0$. Polarisation of \eqref{differential_s} with $X=JW$ concludes the proof.
\end{proof}
Let now $\Ric_g$ be the Ricci curvature $(2,0)$-tensor field of $g$. 
\begin{theorem}
Nearly K\"ahler six-manifolds are Einstein with positive scalar curvature.
\end{theorem}
\begin{proof}
Consider the six-dimensional case in Proposition \ref{main_result_for_einstein}, i.e.\ $n=3$. Identity \eqref{diff_ricci} states that $\Ric-\Ricc = 4\mu^2 \id$, with $\mu>0$ constant. Thus, since $g(E_i,E_j) = \delta_{ij}$,\eqref{main_formula_for_einstein} reduces to
$$\sum R(W,E_i,E_i,X)-5R(W,E_i,JE_i,JX)=0,$$
which is equivalent to saying $\Ric \negthinspace = 5 \Ricc$. Therefore, $\Ric-\Ricc = \Ric-\tfrac{1}{5}\Ric = 4\mu^2 \id$, namely $\Ric_g = 5\mu^2 g$, and $M$ is Einstein with positive scalar curvature.
\end{proof}

\section{Formulation in terms of PDEs}
\label{formulation_in_terms_of_pdes}

We now go through the details behind Theorem \ref{theorem_carrion}, following \cite[Section 4.3]{carrion} for this last part. It will be convenient to work on the complexified tangent bundle $T \otimes \mathbb{C}$ of $M$. We use the standard notations $T^{1,0}$ and $T^{0,1}$ for the eigenspaces of $J$ corresponding to the eigenvalues $i$ and $-i$ respectively, so that $T \otimes \mathbb{C} = T^{1,0} \oplus T^{0,1}$. All linear operations are extended by $\mathbb{C}$-linearity, bars denote complex conjugation.

A first step in the direction we want to take was Proposition \ref{first_equivalent_characterisations}, where we proved that having a nearly K\"ahler structure on $(M,g,J)$ is equivalent to saying $\nabla \sigma$ is a type $(3,0)+(0,3)$ form or that $d\sigma = 3\nabla\sigma$, for $\sigma = g(J{}\cdot{},{}\cdot{})$. We now give further characterisations.
\begin{lemma}
\label{other_equivalent_characterisations}
The following assertions hold:
\begin{enumerate}
\item $M$ is nearly K\"ahler if and only if $\nabla_X Y + \nabla_Y X \in T^{1,0}$ for $X,Y \in T^{1,0}$.
\item If $M$ is nearly K\"ahler then $\nabla_{\overline{X}} Y \in T^{1,0}$, for $X,Y \in T^{1,0}$.
\end{enumerate}
\end{lemma}

\begin{proof}
For $X,Y \in T^{1,0}$ we have
\begin{align*}
J(\nabla_X Y + \nabla_Y X) & = \nabla_X JY + \nabla_Y JX - (\nabla_X J)Y - (\nabla_Y J)X \\
					& =  i(\nabla_X Y + \nabla_Y X) - (\nabla_X J)Y - (\nabla_Y J)X,
\end{align*}
from which our first claim follows. The second is a plain check that $(\nabla_{\overline{X}}J)Y=0$ using the definition of $(1,0)$-vector fields, as $J\nabla_{\overline{X}}Y = \nabla_{\overline{X}}JY-(\nabla_{\overline{X}}J)Y = i\nabla_{\overline{X}}Y-(\nabla_{\overline{X}}J)Y$: for if $X = A-iJA, Y = B-iJB$, then $(\nabla_{\overline{X}}J)Y$ equals
\begin{equation*}
(\nabla_{A+iJA}J)(B-iJB) = (\nabla_AJ)B-i(\nabla_AJ)JB + i(\nabla_{JA}J)B+(\nabla_{JA}J)JB = 0,
\end{equation*}
as it follows by applying \eqref{covariant_derivative_j}.
\end{proof}

\begin{lemma}
\label{prop1}
Let us consider $\{F_i\}_{i=1,2,3}$, a local orthonormal basis of $T^{1,0}$ on \nolinebreak{$M$.} Denote by $\{f^i \}_{i=1,2,3}$ its dual in $\Lambda^{1,0}$. The following facts are equivalent:
\begin{enumerate}
\item \label{1.4.1} $\nabla_X Y + \nabla_Y X \in T^{1,0}$ for $X,Y \in T^{1,0}$.
\item \label{1.4.2} There exists a constant, complex-valued function $\lambda$ such that $[F_i,F_j]^{0,1} = -\overline{\lambda} \overline{F}_k$, where $(i,j,k)$ is a cyclic permutation of $(1,2,3)$.
\item \label{1.4.3} There exists a constant, complex-valued function $\lambda$ such that $(df^i)^{0,2} = \lambda \overline{f}{}^j \wedge \overline{f}{}^k$, where $(i,j,k)$ is a cyclic permutation of $(1,2,3)$.
\end{enumerate}
\end{lemma}

\begin{proof}
Let us prove that \ref{1.4.2} and \ref{1.4.3} are equivalent first. Suppose $(df^i)^{0,2} = \lambda \overline{f}{}^j \wedge \overline{f}{}^k$ for some constant $\lambda \in \mathbb{C}$. Since type $(0,1)$ forms vanish on $(1,0)$ vectors and $(d\overline{f}{}^k)^{2,0} = \overline{(df^k)^{0,2}} = \overline{\lambda} f^i \wedge f^j$, we get
\begin{align*}
[F_i,F_j]^{0,1} & = \sum_{ k=1}^{3} \overline{f}{}^k([F_i,F_j])\overline{F}_k = \sum_{ k=1}^{3} (F_i(\overline{f}{}^k(F_j)) -F_j(\overline{f}{}^k(F_i))-d\overline{f}{}^k(F_i,F_j))\overline{F}_k \\
			& = -\sum_{ k=1}^{3} (d\overline{f}{}^k)^{2,0}(F_i,F_j)\overline{F}_k = -\overline{\lambda} \overline{F}_k.
\end{align*}

Conversely, assume $[F_i,F_j]^{0,1} = -\overline{\lambda} \overline{F}_k$ holds for some complex constant $\lambda$. If we set $$\overline{X} = \sum_{ \ell=1}^{ 3} a_{\ell}\overline{F}_{\ell}, \quad \overline{Y} = \sum_{\ell=1}^{3} b_{\ell} \overline{F}_{\ell},$$ using that $[\overline{F}_j,\overline{F}_k]^{1,0} = \overline{[F_j,F_k]^{0,1}} = -\lambda F_i$ we have
\begin{align*}
\lambda \overline{f}{}^j \wedge \overline{f}{}^k (\overline{X},\overline{Y}) & = \lambda(a_jb_k-b_ja_k) \\
& = -f^i \Bigl(\cyclicsum_{ i,j,k}(a_jb_k - b_ja_k)(-\lambda F_i) \Bigr) \\
& = -f^i \Bigl(\sum_{ j<k} (a_jb_k - b_ja_k)([\overline{F}_j,\overline{F}_k]^{1,0}) \Bigr) \\
& = -f^i([\overline{X},\overline{Y}]^{1,0}) \\
& = -\overline{X}(f^i(\overline{Y})) +\overline{Y}(f^i(\overline{X})) + df^i(\overline{X},\overline{Y}) \\
& = df^i (\overline{X},\overline{Y}).
\end{align*}
This yields our first equivalence.

Let us assume now that $[F_i,F_j]^{0,1} = -\overline{\lambda}\overline{F}_k$ for $\lambda \in \mathbb{C}$. We use that $g(\nabla_{F_j}F_k,F_k)=0$ to compute $g(\nabla_{F_1}F_2+\nabla_{F_2}F_1,F_i)$ for all $i=1,2,3$. We have 
\begin{align*}
g(\nabla_{F_1}F_2+\nabla_{F_2}F_1,F_1) & = g(\nabla_{F_1}F_2-\nabla_{F_2}F_1,F_1) \\
& = g([F_1,F_2],F_1) = -\overline{\lambda}g(\overline{F}_3,F_1) = 0. \\
g(\nabla_{F_1}F_2+\nabla_{F_2}F_1,F_2) & = g(-\nabla_{F_1}F_2+\nabla_{F_2}F_1,F_2) \\
& = g([F_2,F_1],F_2) = \overline{\lambda}g(\overline{F}_3,F_2) = 0. \\
g(\nabla_{F_1}F_2+\nabla_{F_2}F_1,F_3) & = g(\nabla_{F_1}F_2,F_3)+g(\nabla_{F_2}F_1,F_3) \\
& = -g(F_2,\nabla_{F_1}F_3)-g(F_1,\nabla_{F_2}F_3).
\end{align*}
Now note that $g(\nabla_{F_2}F_3-\nabla_{F_3}F_2,F_1) = g(\nabla_{F_3}F_1-\nabla_{F_1}F_3,F_2) = -\overline{\lambda}$. This yields 
\begin{align*}
-g(F_2,\nabla_{F_1}F_3)-g(F_1,\nabla_{F_2}F_3) & = -g(F_2,\nabla_{F_1}F_3)+\overline{\lambda}-g(F_1,\nabla_{F_3}F_2) \\
& = g(F_2,\nabla_{F_3}F_1-\nabla_{F_1}F_3)+\overline{\lambda} \\
& = -\overline{\lambda}+\overline{\lambda} = 0.
\end{align*}
The other cases are analogous and \ref{1.4.1} follows.

Finally, we prove that \ref{1.4.1} implies \ref{1.4.2}. Assuming $\nabla_X Y + \nabla_Y X \in T^{1,0}$ with $X,Y \in T^{1,0}$, we have 
$$g([F_i,F_j]^{0,1},F_k) = g([F_i,F_j],F_k) = g(\nabla_{F_i}F_j,F_k) - g(\nabla_{F_j}F_i,F_k) = 2g(\nabla_{F_i}F_j,F_k),$$
as the metric is of type $(1,1)$ and $\nabla_{F_i}F_j = -\nabla_{F_j}F_i + W,\, W \in T^{1,0}$ by assumption. The basis has type $(1,0)$, so
\begin{align*}
g(\nabla_{F_i}F_j, F_k) & = g(J\nabla_{F_i}F_j, JF_k) \\
& = g(\nabla_{F_i}JF_j - (\nabla_{F_i}J)F_j,JF_k) \\
& = -g(\nabla_{F_i}F_j, F_k) - g((\nabla_{F_i}J)F_j,iF_k),
\end{align*}
which implies $2g(\nabla_{F_i}F_j,F_k) = -i\nabla\sigma(F_i,F_j,F_k)$, and $g(\nabla_{F_i}F_j,F_k)$ is totally skew-symmetric in $i,j,k$. So we can write it as
\begin{equation*}
2g(\nabla_{F_i}F_j,F_k) = -\varepsilon_{ijk}\overline{\lambda}
\end{equation*}
for some complex valued function $\lambda$ on $M$, where $\varepsilon_{ijk}$ is the sign of the permutation $(i,j,k)$ and takes value $0$ when any two indices coincide. There remains to prove that $\lambda$ is constant. To this aim, take any real, local orthonormal set $\{E_1,JE_1,E_2,JE_2\}$. We put $(\nabla_{E_1}J)E_2 \coloneqq \mu E_3$, where $E_3$ is a unit vector and $\mu$ a non-negative real function satisfying \eqref{mu_local_to_global}. Then set $F_k \coloneqq (1/\negthinspace \sqrt{2})(E_k-iJE_k)$ in $T^{1,0}, k = 1,2,3$, and recall that $g(\overline{X},\overline{Y}) = \overline{g(X,Y)}$ and $\overline{\nabla_X Y} = \nabla_{\overline{X}}\overline{Y}$ for every $X,Y \in T \otimes \mathbb{C}$. Hence $-\lambda = 2g(\nabla_{\overline{F}_1}\overline{F}_2,\overline{F}_3)$. Here below we find the relationship between $\lambda$ and $\mu$:
\begin{align*}
-\lambda & = 2g(\nabla_{\overline{F}_1}\overline{F}_2,\overline{F}_3) \\
	       & = g(\nabla_{E_1}E_2 - \nabla_{JE_1}JE_2 + i(\nabla_{E_1}JE_2+\nabla_{JE_1}E_2),\overline{F}_3) \\
	       & = g(\nabla_{E_1}E_2 + iJ\nabla_{E_1}E_2,\overline{F}_3) + ig(\nabla_{JE_1}E_2+iJ\nabla_{JE_1}E_2,\overline{F}_3) \\
	       & \qquad + g(J(\nabla_{E_1}J)E_2 + i(\nabla_{E_1}J)E_2,\overline{F}_3). 
\end{align*}
Observe that $\nabla_{E_1}E_2 + iJ\nabla_{E_1}E_2$ and $\nabla_{JE_1}E_2+iJ\nabla_{JE_1}E_2$ are of type $(0,1)$, so the first two terms vanish, and expanding the last term we find
\begin{equation*}
g(J(\nabla_{E_1}J)E_2 + i(\nabla_{E_1}J)E_2,\overline{F}_3) = i\mu g(E_3-iJE_3,\overline{F}_3) = i\sqrt{2}\mu.
\end{equation*}
In Proposition \ref{mu_constant} we proved that $\mu$ is constant, so $\lambda$ is constant as well.
\end{proof}

\begin{theorem}
\label{modern_nK_conditions}

Let $(M,g,J)$ be an almost Hermitian six-manifold. Then $M$ is nearly K\"ahler if and only if there exist a complex three-form $\psi_{\mathbb{C}} = \psi_+ + i\psi_-$ and a constant function $\mu$ such that 
\begin{equation}
\label{first_nk_structure_equations}
d\sigma = 3\mu \psi_+, \qquad d\psi_- = -2\mu\sigma \wedge \sigma.
\end{equation}
\end{theorem}

\begin{proof}
Assume $M$ is nearly K\"ahler. Using the local orthonormal basis as in Lemma \ref{prop1}, we can write locally $\sigma = i\sum_{ k=1}^{3} f{}^k \wedge \overline{f}{}^k$, where $f^k = (1/\negthinspace \sqrt{2})(e^k+iJe^k)$. Let us define 
\begin{equation*}
\psi_{\mathbb{C}} = \psi_+ + i\psi_- \coloneqq 2\sqrt{2}f^1 \wedge f^2 \wedge f^3.
\end{equation*}
We know by Proposition \ref{first_equivalent_characterisations} that $M$ nearly K\"ahler implies $d\sigma \in [\![\Lambda^{3,0}]\!]$. We thus calculate its $(3,0)+(0,3)$ part.
\begin{equation*}
(id\sigma)^{3,0} = -\Bigl(\sum_{ k=1}^{3} df^k \wedge \overline{f}{}^k - f^k \wedge d\overline{f}{}^k\Bigr)^{3,0} = \sum_{ k=1}^{3} (f^k \wedge d\overline{f}{}^k)^{3,0} = \sum_{ k=1}^{3} f^k \wedge (d\overline{f}{}^k)^{2,0}.
\end{equation*}
Lemma \ref{prop1} implies
\begin{align*}
(id\sigma)^{3,0} & = \sum_{ k=1}^{3} f^k \wedge (d\overline{f}{}^k)^{2,0} = \overline{\lambda} \cyclicsum_{ 1,2,3}f^1 \wedge f^2 \wedge f^3 = \tfrac{3}{2\sqrt{2}}\overline{\lambda} \psi_{\mathbb{C}}.
\end{align*}
Similarly, $(id\sigma)^{0,3} = -(3/2\sqrt{2})\lambda \overline{\psi_{\mathbb{C}}}$. We found $\lambda = -i\sqrt{2}\mu$, so $id\sigma = (id\sigma)^{3,0}+(id\sigma)^{0,3} = 3i\mu\psi_+$.
This implies $0 = d\psi_+$, hence $d\psi_{\mathbb{C}} = -d\overline{\psi}_{\mathbb{C}}$. Differentiating $\psi_{\mathbb{C}}$ we find
\begin{align}
\label{differential_psi_c}
d\psi_{\mathbb{C}} & = 2\sqrt{2}\bigl(df^1 \wedge f^2 \wedge f^3 - f^1 \wedge df^2 \wedge f^3 + f^1 \wedge f^2 \wedge df^3\bigr) \nonumber \\
& = 2\sqrt{2}\bigl((df^1)^{1,1} \wedge f^2 \wedge f^3 + (df^1)^{0,2} \wedge f^2 \wedge f^3 \nonumber \\ 
& \qquad - f^1 \wedge (df^2)^{1,1} \wedge f^3 -f^1 \wedge (df^2)^{0,2} \wedge f^3 \nonumber \\
& \qquad +f^1 \wedge f^2 \wedge (df^3)^{1,1} + f^1 \wedge f^2 \wedge (df^3)^{0,2}\bigr) \in \Lambda^{3,1} + \Lambda^{2,2}.
\end{align}
With similar computations one can see that $d\overline{\psi}_{\mathbb{C}} \in \Lambda^{1,3} + \Lambda^{2,2}$. We proved that $d\psi_{\mathbb{C}} = -d\overline{\psi}_{\mathbb{C}}$, so the $(3,1)$ part of $d\psi_{\mathbb{C}}$ vanishes. We then have
\begin{equation}
\label{psi_-_sigma_squared}
id\psi_- = 2\sqrt{2}\lambda \sum_{ j<k}\overline{f}{}^j \wedge \overline{f}{}^k \wedge f^j \wedge f^k = -2i\mu\sigma \wedge \sigma
\end{equation}
and the first implication is done.

Conversely, given $d\sigma = 3\mu \psi_+$ and $d\psi_- = -2\mu \sigma \wedge \sigma$, it is enough to prove that $(df^i)^{0,2}= \lambda \overline{f}{}^j \wedge \overline{f}{}^k$ for $(i,j,k)$ cyclic permutation of $(1,2,3)$ and some constant $\lambda \in \mathbb{C}$. To get it, we first see that
\begin{equation*}
\psi_{\mathbb{C}} \wedge (df^i)^{0,2} = \psi_{\mathbb{C}} \wedge df^i = d\psi_{\mathbb{C}} \wedge f^i = id\psi_- \wedge f^i = \psi_{\mathbb{C}} \wedge \lambda(\overline{f}{}^j \wedge \overline{f}{}^k).
\end{equation*}
Now observe that the map $\Lambda^{0,2} \rightarrow \Lambda^{3,2}$ given by the wedge product with $\psi_{\mathbb{C}}$ is injective. This implies $(df^i)^{0,2} = \lambda \overline{f}{}^j \wedge \overline{f}{}^k$.

Lastly, note that the form $\psi_{\mathbb C}$ is globally defined (cf.\ \cite{bryant_2}, Section 2), and we can then conclude.
\end{proof}

\begin{remark}
\label{alternative_proof}
We can give a proof of the first implication in Theorem \ref{modern_nK_conditions} without having $\mu$ constant a priori.
Assume $(M,g,J)$ is nearly K\"ahler. 
Note that in the proof of Lemma \ref{prop1} we use Proposition~\ref{mu_constant} to show that $\lambda$ is constant only at the end. Without relying on the latter result,
 Lemma~\ref{prop1} (in particular point~\ref{1.4.3}) holds without the constancy of~$\lambda$. The relation between 
$\lambda$ and $\mu$ is unchanged. 
Now the initial part of Theorem \ref{modern_nK_conditions} holds (except the constancy of $\lambda$) and yields $d\sigma=3\mu \psi_+$,
where $\psi_+$ is the real part of the (only locally defined) $(3,0)$-form $\psi_{\mathbb C}$ introduced at the beginning of the proof of the same theorem,
and $\mu$ is a not identically zero function. 
Clearly $d(\mu\psi_+)=0$, so $d(\mu\psi_{\mathbb C}) = id(\mu \psi_-) = -d(\mu \overline{\psi_{\mathbb C}})$.
By the same calculation as in \eqref{differential_psi_c}, $d(\mu \psi_{\mathbb C})$ turns out to be of type $(2,2)$, hence the 
only non-vanishing part of $d(\mu\psi_-)$ is the $(2,2)$-part. Now observe that
$(d\mu \wedge \psi_-)^{2,2}=0$ by $(d\mu \wedge \psi_{\mathbb C})^{2,2}=0$. Therefore, identity \eqref{psi_-_sigma_squared} implies
\begin{equation*}
d(\mu\psi_-)=d(\mu\psi_-)^{2,2} = (d\mu \wedge \psi_- + \mu d\psi_-)^{2,2} = \mu(d\psi_-)^{2,2} = -2\mu^2\sigma^2.
\end{equation*}
But then, since $\sigma \wedge \psi_+=0$, we find
\begin{align*}
0 = d^2(\mu \psi_-) & = d\mu^2 \wedge \sigma^2+2\mu^2\sigma \wedge d\sigma \\
& = d\mu^2 \wedge \sigma^2 + 6\mu^3\sigma \wedge \psi_+ \\
& = d\mu^2 \wedge \sigma^2,
\end{align*}
whence $\mu$ is locally constant, as $\sigma$ is non-degenerate. Connectedness of $M$ implies that $\mu$ is globally constant.
This proves $d\sigma = 3\mu\psi_+$ and $d\psi_- = -2\mu \sigma \wedge \sigma$ with $\mu$ constant. Now the fact that nearly K\"ahler six-manifolds are Einstein follows from the connection with $\mathrm{G}_2$-holonomy
cones, which are Ricci-flat, or from the computation of the Ricci tensor for $\mathrm{SU}(3)$-structures, see Bedulli--Vezzoni \cite{bedulli_vezzoni}.
\end{remark}

\begin{remark}
\label{two_equivalent_definitions}
We can rescale our basis so that $\sigma \mapsto \widetilde{\sigma}\coloneqq \mu^2\sigma$ and $\psi_{\pm} \mapsto \widetilde{\psi}_{\pm} \coloneqq \mu^3\psi_{\pm}$. Then 
\begin{equation*}
d\widetilde{\sigma} = 3\widetilde{\psi}_+, \qquad 
d\widetilde{\psi}_- = -2\widetilde{\sigma} \wedge \widetilde{\sigma}.
\end{equation*}
\end{remark}
As already discussed in Remark \ref{alternative_proof}, Theorem \ref{modern_nK_conditions} provides us with a characterisation of nearly K\"ahler six-manifolds in terms of an $\mathrm{SU}(3)$-structure.
\begin{definition}
\label{modern_def_nk}
Let $(M,g,J)$ be an almost Hermitian six-dimensional manifold with an $\mathrm{SU}(3)$-structure $(\sigma = g(J{}\cdot{},{}\cdot{}),\psi_{\mathbb{C}}=\psi_++i\psi_-)$. We say that $M$ is \emph{nearly K\"ahler} if and only if 
\begin{equation*}
d\sigma = 3\psi_+, \qquad d\psi_- = -2\sigma \wedge \sigma,
\end{equation*}
up to homothety.
\end{definition}
Observe that locally $\sigma$ and $\psi_{\mathbb{C}}$ were expressed in terms of type $(1,0)$ vectors $f^i$ as $\sigma = i\sum_{ k=1}^{3} f^k \wedge \overline{f}{}^k$ and $\psi_{\mathbb{C}} = 2\sqrt{2} f^1\wedge f^2 \wedge f^3$, thus giving the real models
\begin{align}
\label{sigma_model} \sigma & = e^1\wedge Je^1 + e^2\wedge Je^2 + e^3\wedge Je^3, \\
\label{psi_+_model} \psi_+ & = e^1 \wedge e^2 \wedge e^3 - Je^1 \wedge Je^2 \wedge e^3 - e^1\wedge Je^2 \wedge Je^3 - Je^1 \wedge e^2 \wedge Je^3, \\
\label{psi_-_model} \psi_- & = e^1 \wedge e^2 \wedge Je^3 - Je^1 \wedge Je^2 \wedge Je^3 + e^1\wedge Je^2 \wedge e^3 + Je^1 \wedge e^2 \wedge e^3,
\end{align}
which are obtained by the definition $f^k \coloneqq (1/\negthinspace \sqrt{2})(e^k+iJe^k),k=1,2,3$. By $Je^i = -e^i \circ J$, expressions \eqref{psi_+_model} and \eqref{psi_-_model}, the relation $\psi_- = -\psi_+({}\cdot{},{}\cdot{},J{}\cdot{})$ readily follows. On the other hand by equations \eqref{first_nk_structure_equations} and Proposition \ref{first_equivalent_characterisations} we have $d\sigma = 3\mu \psi_+ = 3\nabla \sigma$, so $\mu \psi_+ = \nabla\sigma$. But since $\nabla \sigma \in \Lambda^1 \otimes [\![\Lambda^{2,0}]\!]$ we find
\begin{align*}
\psi_-(X,Y,Z) & = -\psi_+(X,Y,JZ) = -\mu^{-1}\nabla\sigma(X,Y,JZ) = -\mu^{-1}\nabla\sigma(JZ,X,Y) \\
& = \mu^{-1}\nabla\sigma(JZ,JX,JY) = \mu^{-1}\nabla\sigma(JX,JY,JZ) = -J\psi_+(X,Y,Z).
\end{align*}
Therefore $\psi_- = -J\psi_+$. 
\begin{remark}
To derive $\psi_- = -J\psi_+$ we have used the nearly K\"ahler structure on $M$ for convenience. However, such identity is 
not special of our set-up and can be simply derived by $\mathrm{SU}(3)$-linear algebra as in \cite{chiossi_salamon}.
\end{remark}
\begin{remark}
\label{hodge_star_relation_volume}
Let us set $\mathrm{vol} \coloneqq e^1 \wedge Je^1 \wedge e^2 \wedge Je^2 \wedge e^3 \wedge Je^3$. A straightforward calculation of $\psi_+ \wedge \psi_-$ and $\sigma \wedge \psi_{\pm}$ gives 
\begin{equation*}
\psi_+\wedge \psi_- = 4\mathrm{vol} = \tfrac{2}{3}\sigma^3, \qquad \sigma \wedge \psi_{\pm} = 0.
\end{equation*}
Since $g(\psi_+,\psi_+) = 4$ the first equation tells us that $\psi_+ \wedge \psi_- = 4\mathrm{vol} = g(\psi_+,\psi_+)\mathrm{vol} = \psi_+ \wedge {*}\psi_+$, so by uniqueness of ${*}\psi_+$ we deduce ${*}\psi_+ = \psi_-$. 
\end{remark}

Recall that $\widehat{\nabla} \coloneqq \nabla - \tfrac{1}{2}J(\nabla J)$ is a $\mathrm{U}(3)$-connection by Proposition \ref{unitary_connection}. 
\begin{prop}
\label{su3_connection}
$\widehat{\nabla}$ is an $\mathrm{SU}(3)$-connection.
\end{prop}
\begin{proof}
We calculate $\widehat{\nabla}(\nabla\sigma)$. By \eqref{cyclic_sum_sigma} we have
\begin{align*}
\widehat{\nabla} (\nabla\sigma)(W,X,Y,Z) & = W(\nabla\sigma(X,Y,Z)) - \nabla\sigma(\widehat{\nabla}_WX,Y,Z) \\
& \qquad -\nabla\sigma(X,\widehat{\nabla}_W Y,Z) -\nabla \sigma(X,Y,\widehat{\nabla}_WZ) \\
& = \nabla^2\sigma(W,X,Y,Z)+\tfrac{1}{2}\cyclicsum_{ X,Y,Z} g((\nabla_WJ)X,(\nabla_YJ)JZ) = 0,
\end{align*}
which proves $\widehat{\nabla}(\nabla\sigma) = 0 = \widehat{\nabla}\psi_+$, thus $\psi_+$ is parallel. Further, by $\psi_- = -J\psi_+$ we have at once $\widehat{\nabla}\psi_- = 0$, namely $\widehat{\nabla}\psi_{\mathbb{C}}=0$, which proves $\widehat{\nabla}$ is actually an $\mathrm{SU}(3)$-connection. 
\end{proof}

\begin{remark}
\label{last_remark}
We mentioned already in Proposition \ref{first_equivalent_characterisations} that $\nabla\sigma$ lies in $[\![\Lambda^{3,0}]\!]$, so obviously it is only the $(3,0)+(0,3)$ part of $\nabla\sigma$ that measures the failure of $M$ to be K\"ahler. Therefore, we can say that it is exactly the type of $\nabla\sigma$ that determines the class of nearly K\"ahler manifolds in the classification completed by Gray and Hervella. On the other hand, $\nabla\sigma$ may be identified with $\nabla J$ (via $\nabla \sigma(X,Y,Z) = g((\nabla_X J)Y,Z)$), which may in turn be identified with the Nijenhuis tensor $N$ of $J$ by Proposition \ref{nijenhuis_nk}. The latter is the \emph{intrinsic torsion} of the $\mathrm{SU}(3)$-structure $(\sigma,\psi_{\pm})$ by Proposition~\ref{su3_connection}. A detailed study of this object for $\mathrm{SU}(3)$- and $\mathrm{G}_2$-structures was pursued by Chiossi and Salamon (see~\cite{chiossi_salamon}, in particular Theorem 1.1 for what regards our set-up).
\end{remark}

%\bibliographystyle{amsplain}
%\bibliography{bibliography}

\begin{thebibliography}{10}

\bibitem{bar}
C.~B\"ar, \emph{Real Killing spinors and holonomy}. Comm.\ Math.\ Phys.\ (1993) \textbf{154}, 509--521.

\bibitem{bedulli_vezzoni}
L.~Bedulli, L.~Vezzoni, \emph{The Ricci tensor of $\mathrm{SU}(3)$-manifolds}, J.\ Geom.\ Phys.\ (2007) \textbf{57}, no.\ 4, 1125--1146.

\bibitem{bryant}
R.~L.~Bryant, \emph{Metrics with exceptional holonomy}, Annals of Mathematics (2) (1987) \textbf{126}, no.\ 3, pp.\ 525--576.

\bibitem{bryant_2}
R.~L.~Bryant, \emph{On the geometry of almost complex $6$-manifolds}, Asian J.~Math (2006) \textbf{10}, no.\ 3, pp.\ 561--606.

\bibitem{butruille}
J.-B.~Butruille, \emph{Classification des vari\'et\'es approximativement k\"ahleriennes homog\`enes}, Ann.\ Global Anal.\ Geom. (2005) \textbf{27}, no.\ 3, pp.\ 201--225.

\bibitem{carrion}
R.~R.~Carri\'on, \emph{Some special geometries defined by Lie groups}, PhD thesis, Oxford, 1993.

\bibitem{chiossi_salamon}
S.~Chiossi, S.~M.~Salamon, \emph{The intrinsic torsion of $\mathrm{SU}(3)$ and $\mathrm{G}_2$ structures}, Differential geometry, Valencia (2001), pp.\ 115--133, World Sci.\ Publ., River Edge, NJ, 2002. 

\bibitem{foscolo_haskins}
L.~Foscolo, M.~Haskins, \emph{New $\mathrm{G}_2$-holonomy cones and exotic nearly K\"ahler structures on $\mathbb{S}^6$ and $\mathbb{S}^3 \times \mathbb{S}^3$}, Annals of Mathematics (2) (2017) \textbf{185}, no.\ 1, pp.\ 59--130.

\bibitem{friedrich_grunewald}
T.~Friedrich, R.~Grunewald, \emph{On the first eigenvalue of the Dirac operator on $6$-dimensional manifolds}, Ann.\ Global Anal.\ Geom.\ (1985) \textbf{3}, no.\ 3. pp.\ 265--273.

\bibitem{frolicher}
A.~Fr\"olicher, \emph{Zur Differentialgeometrie der komplexen Strukturen}, Mathematische Annalen (1955) \textbf{129}, no.\ 1, pp.\ 50--95.

\bibitem{fukami_ishihara}
T.~Fukami, S.~Ishihara, \emph{Almost Hermitian structure on $\mathbb{S}^6$}, T\^ohoku Math. J. (2) (1955) \textbf{7}, no.\ 3, pp.\ 151--156.

\bibitem{gray}
A.~Gray, \emph{Nearly K\"ahler manifolds}, J.\ Differential Geometry (1970)
  \textbf{4}, no.\ 3, pp.\ 283--309.

\bibitem{gray1976}
\bysame, \emph{The structure of nearly K\"ahler manifolds}, Math. Ann. (1976)
  \textbf{223}, no.\ 3, pp.\ 233--248.
  
\bibitem{gray1969}
\bysame, \emph{Almost complex submanifolds of the six sphere}, Proceedings of the American Mathematical Society (1969) \textbf{20}, no.\ 1, pp.\ 277--279.

\bibitem{gray_hervella}
A.~Gray, L.~M.~Hervella, \emph{The sixteen classes of almost Hermitian manifolds and their linear invariants}, Annali di Matematica pura ed applicata (1980) \textbf{123}, no.\ 1, pp.\ 35--58. 

\bibitem{gray_wolf}
A.~Gray, J.~A.~Wolf, \emph{Homogeneous spaces defined by Lie group automorphisms}, J.\ Differential Geom. (1968) \textbf{2}, no.\ 1, pp.\ 77--114.

\bibitem{grunewald}
R.~Grunewald, \emph{Six-dimensional Riemannian manifolds with a real Killing spinor}, Ann.\ Global Anal.\ Geom.\ (1990) \textbf{8}, no.\ 1, pp.\ 43--59.

\bibitem{morris}
D.~Morris, \emph{Nearly K\"ahler geometry in six dimensions}, MPhil thesis, London, 2014.
  
\bibitem{nagy}
P.-A.~Nagy, \emph{Nearly K\"ahler geometry and Riemannian foliations}, Asian J.\ Math. (2002) \textbf{6}, no.\ 3, pp.\ 481--504.  

\bibitem{phd_thesis}
G.~Russo, \emph{Torus symmetry and nearly K\"ahler metrics}, PhD thesis, Aarhus, 2019.

\bibitem{russo_swann}
G.~Russo, A.~Swann, \emph{Nearly K\"ahler six-manifolds with two-torus symmetry}, Journal of Geometry and Physics (2019) \textbf{138}, pp.\ 144--153.

\bibitem{salamon}
S.~M.~Salamon, \emph{Riemannian geometry and holonomy groups}, Pitman Research Notes in Mathematics Series, 201. Longman Scientific \& Technical, Harlow (1989).
  
\end{thebibliography}
\providecommand{\bysame}{\leavevmode\hbox to3em{\hrulefill}\thinspace}
\providecommand{\MR}{\relax\ifhmode\unskip\space\fi MR }
\providecommand{\MRhref}[2]{%
  \href{http://www.ams.org/mathscinet-getitem?mr=#1}{#2}
}
\providecommand{\href}[2]{#2}

(G.~Russo), Max Planck Institute for Mathematics, Vivatsgasse 7, 53111 Bonn, Germany.

\textit{E-mail address}: \texttt{giovanni.russo@math.au.dk}

\end{document}